\documentclass[english, a4paper, 12pt]{article}
\usepackage{multicol,amsmath,t1enc,a4,babel}
\usepackage{amsthm}
\usepackage{amssymb}
\usepackage[dvips]{graphicx}
\usepackage[dvips]{epsfig}
\usepackage{times}
\usepackage{amsfonts}
\usepackage{wasysym}
\usepackage{mathtools}
\usepackage[colorinlistoftodos]{todonotes}
\usepackage{mathrsfs}
\usepackage{relsize}
\usepackage{bbm}

\newcommand{\e}{\mathbb{E}}

\newcommand{\p}{\mathbb{P}}

\newcommand{\h}{{H}}

\newcommand{\ind}{\mathbbm{1}_{A_T}}
\newcommand{\indi}{\mathbbm{1}_{A_T^c}}
\DeclareMathOperator{\vect}{vec}
\DeclareMathOperator{\rank}{rank}
\newcommand{\bh}{{\Theta}}
\newcommand{\Bh}{{B}}
\newcommand{\Ch}{{C}}
\newcommand{\Dh}{{D}}

\newcommand{\gammah}{{\gamma}}
\newtheorem{defi}{Definition}[section]
\newtheorem{lemma}[defi]{Lemma}
\newtheorem{prop}[defi]{Proposition}
\newtheorem{theorem}[defi]{Theorem}
\newtheorem{kor}[defi]{Corollary}

\newtheorem{rem}[defi]{Remark}

\begin{document}

\title{Vector-valued Generalised Ornstein-Uhlenbeck Processes}

\renewcommand{\thefootnote}{\fnsymbol{footnote}}

\author{Marko Voutilainen\footnotemark[1],\, Lauri Viitasaari\footnotemark[2],\, Pauliina Ilmonen\footnotemark[1],\,\\ Soledad Torres\footnotemark[3],\, and \, Ciprian Tudor\footnotemark[4]}

\footnotetext[1]{Department of Mathematics and Systems Analysis, Aalto University School of Science, Finland}

\footnotetext[2]{Department of Information and Service Management, Aalto University School of Business, Finland}

\footnotetext[3]{CIMFAV, Facultad de Ingenier\'ia, Universidad de Valpara\'iso,
 Valparaiso, Chile}

\footnotetext[4]{UFR Math\'ematiques, Universit\'e de Lille 1, France}

\maketitle

\begin{abstract}
\noindent
Generalisations of the Ornstein-Uhlenbeck process defined through Langevin equation $dU_t = - \bh U_t dt + dG_t,$ such as fractional Ornstein-Uhlenbeck processes, have recently received a lot of attention in the literature. In particular, estimation of the unknown parameter $\bh$ is widely studied under Gaussian stationary increment noise $G$. Langevin equation is well-known for its connections to physics. In addition to that, motivation for studying Langevin equation with a general noise $G$ stems from the fact that the equation characterises all univariate stationary processes. Most of the literature on the topic focuses on the one-dimensional case with Gaussian noise $G$. In this article, we consider estimation of the unknown model parameter in the multidimensional version of the Langevin equation, where the parameter $\bh$ is a matrix and $G$ is a general, not necessarily Gaussian, vector-valued process with stationary increments. Based on algebraic Riccati equations, we construct an estimator for the matrix $\bh$. Moreover, we prove the consistency of the estimator and derive its limiting distribution under natural assumptions. In addition, to motivate our work, we prove that the Langevin equation characterises all stationary processes in a multidimensional setting as well.
\end{abstract}

{\small
\medskip


\medskip

}

%%%%%%%%%%%%%%%%%%%%%%%%%%%%%%%%%%%%%%%%%%%%%%%%%%%%%%%%%%%%%%%%%%%%%%%%%%%%%%%

%\tableofcontents
%%%%%%%%%%%%%%%%%%%%%%%%%%%%%%%%%%%%%%%%%%%%%%%%%%%%%%%%%%%%%%%%%%%%%%%%%%%%%%%
\section{Introduction}
In this article, we study statistical problems related to the multidimensional version of generalised Langevin equation
\begin{equation}\label{eq:langevin}
d U_t = -\bh U_t \,d t + d G_t, \qquad t\in\mathbb{R},
\end{equation}
with some stationary increment noise $G$. Here $\bh$ is a positive definite matrix, and the noise $G$ and the solution $U$ are understood as random vectors.

A classical Ornstein-Uhlenbeck process can be defined through the Langevin equation
\begin{equation}
\label{eq:langevin-1d}
d U_t = - \theta U_t \, dt + d W_t, \qquad t \in \mathbb{R},
\end{equation}
where $\theta>0$ is a parameter and $W$ is a Brownian motion. The stationary Ornstein-Uhlenbeck process $U$ can be obtained by a suitable choice of the initial condition $U_0$. Such equations have connections especially to physics, and this is part of the reason why Langevin equations of the form \eqref{eq:langevin} have received a lot of attention in the literature. In addition to the connections to physics, it was recently proven in \cite{Viitasaari-2016a} that, in one dimensional case, Langevin equations characterise essentially all stationary processes (for analogous results in discrete time, we refer to \cite{Voutilainen-Viitasaari-Ilmonen-2017,Voutilainen-Viitasaari-Ilmonen-2019}). This highlights the importance of \eqref{eq:langevin} even further.

The Equation \eqref{eq:langevin} is well-motivated, and there is a vast array of literature studying it with varying driving force $G$. One natural generalisation is to replace the Brownian motion with a L\'{e}vy process. In the infinite-dimensional case, Equation \eqref{eq:langevin-1d} driven by a L\'{e}vy-process has connections to different branches of probability theory such as stochastic partial differential equations, branching processes, generalised Mehler semigroups, and self-decomposable distributions. For a recent survey on the topic, we refer to \cite{applebaum}.

Another natural generalisation is to replace the Brownian motion $W$ in \eqref{eq:langevin-1d} with a fractional Brownian motion $B^H$. (For details on fractional Brownian motion, we refer e.g. to \cite{Mishura-2008}.) The solution $U$, called the \emph{fractional Ornstein-Uhlenbeck process}, was first introduced by \cite{Cheridito-Kawaguchi-Maejima-2003} (see also \cite{Kaarakka-Salminen-2011}). Statistical analysis related to the fractional Ornstein-Uhlenbeck model was initiated in \cite{Hu-Nualart-2010} and \cite{Kleptsyna-LeBreton-2002}, and it has been a very active research topic ever since. Of the studies on parameter estimation in such models, we mention \cite{Azmoodeh-Viitasaari-2015a,bajja-et-al,Ciprian-infinite,Brouste-Iacus-2013,dehling-et-al,Ciprian-Hermite,Es-Sebaiy-Ndiaye-2014,Es-Sebaiy-Tudor-2015,david-general-H,Kozachenko-Melnikov-Mishura-2015,Kubilius-Mishura-Ralchenko-Seleznjev-2015,maslowski-pospisil,Shen-Yin-Yan-2016,Shen-Xu-2014,Sottinen-Viitasaari-2017a,Sun-Guo-2015,Tanaka-2015} to name a few recent ones. Finally, we mention \cite{magdziarz}, that considers fractional extensions of the Levy-driven Ornstein-Uhlenbeck processes.

Despite the vast amount of literature related to \eqref{eq:langevin}, to the best of our knowledge most of it focuses on one-dimensional case, and with a specific driver $G$. In particular, even if the problem is studied in a greater generality to some extent, usually the assumptions are somehow related to the one-dimensional case, or to a specific driver. For example, in \cite{Ciprian-infinite,maslowski-pospisil} the authors studied infinite dimensional fractional Ornstein-Uhlenbeck processes, but there was only one unknown parameter $\bh \in \mathbb{R}$ to estimate and the driver was of a specific form. Similarly, for example in \cite{bajja-et-al,dehling-et-al} there were many parameters to estimate, but again the driver (and the model) was of a specific form. Finally, while in \cite{Sottinen-Viitasaari-2017a} the authors studied a more general noise $G$ that is not related to the fractional Brownian motion, $G$ was still assumed to be Gaussian and the equation was considered only in one dimension with one parameter to estimate. Similarly, in \cite{Ciprian-Hermite}, \cite{Nourdin-Tran},   the authors considered non-Gaussian case, but only with one parameter and a specific, though non-Gaussian, driver.

The aim of this article is to study general multidimensional Langevin equations with arbitrary stationary increment noise. That is, we study \eqref{eq:langevin} with an unknown positive definite matrix $\bh$. We prove that \eqref{eq:langevin} characterises (essentially) all stationary processes, thus giving a natural multidimensional extension of the results presented in \cite{Viitasaari-2016a}. Moreover, given that the underlying processes are square integrable, we provide representations of the cross-covariance matrix $\gamma(t)$ of the stationary solution $U_t$, and show that the unknown $\bh$ solves a certain continuous-time algebraic Riccati equation (CARE). Initiated by the seminal paper \cite{kalman1960contributions} by Kalman, CARE arise naturally in optimal control and filtering theory. As such, we relate the Langevin equation to these fields as well.

We also consider statistical estimation of the unknown matrix $\bh$. Motivated by the relation to CARE, we define the estimator as the solution to a perturbed CARE, in which the coefficient matrices are replaced by estimated ones, and where the cross-covariances $\gamma(t)$ are replaced by their estimators $\hat{\gamma}(t)$. We prove that our estimator is consistent whenever the cross-covariance estimators are consistent. We also study how the rate of convergence and the limiting distribution of our estimator are related to the convergence rate and the limiting distribution of $\hat{\gamma}(t)$.

The rest of the paper is organised as follows. In Section \ref{sec:main} we present and discuss our main results. In particular, we state the characterisation of stationary processes through \eqref{eq:langevin} and we provide the connection to CARE. We also define our estimator for the unknown $\bh$ and provide results on its asymptotic properties. In Subsection \ref{subsection:gaussian} we illustrate the applicability of our results to the Gaussian case. All the proofs are postponed to Section \ref{sec:proofs}.

\section{Multidimensional generalised Langevin equations}
\label{sec:main}
We consider the $n$-dimensional Langevin equation
\begin{equation}
\label{langevin}
dU_t = -\bh U_tdt + dG_t, \qquad t\in\mathbb{R}
\end{equation}
driven by $G = (G_t)_{t\in\mathbb{R}}$,
with a positive definite coefficient matrix $\bh$.
Note that $G$ and the solution $U$ are $n$-dimensional vector-valued processes. Here we understand the solution in the strong sense, with a given initial condition to be specified later. The components of the vectors are denoted by super indices, e.g. $U^{(k)}_t$ denotes the $k$th component of the vector $U_t$, and is a real-valued random process. We denote by $S^n$ the set of symmetric
$n\times n$-matrices, and with the notation $\bh>0$ we mean that the matrix $\bh$ is positive definite. Similarly, by writing $\bh\geq 0$ we mean that $\bh$ is positive semidefinite. If not stated otherwise, we use $\Vert\cdot\Vert$ to denote the standard $L^2$ norm and the corresponding induced matrix norm.
If two processes $(X_t)_{t\in\mathbb{R}}$ and $(Y_t)_{t\in\mathbb{R}}$ have equal finite dimensional distributions, we write $(X_t)_{t\in\mathbb{R}} \overset{law}{=} (Y_t)_{t\in\mathbb{R}}$. Furthermore, with stationary processes we mean
$n$-dimensional strictly stationary processes, i.e. processes for which $(X_{t+s})_{t\in\mathbb{R}} \overset{law}{=} (X_t)_{t\in\mathbb{R}}$ for every $s\in\mathbb{R}$. In addition, we assume that the driver $G$ (and consequently, the solution $U$), have continuous paths almost surely. This guarantees that integrals of type
\begin{equation*}
\int_s^t e^{\bh u} dX_u
\end{equation*}
can be understood componentwise as pathwise Riemann-Stieltjes integrals via integration by parts
\begin{equation*}
\int_s^t e^{\bh u} dX_u = e^{\bh t} X_t - e^{\bh s} X_s - \bh\int_s^t e^{\bh u} X_u du.
\end{equation*}
Indefinite integrals over an interval $[-\infty, t]$ are defined similarly as
\begin{equation}
\label{indefinite}
\int_{-\infty}^t e^{\bh u} dX_u = e^{\bh t}X_t - \bh\lim_{s\to-\infty} \int_s^t e^{\bh u} X_u du
\end{equation}
provided that the limit exists almost surely.

Our first main theorem below shows that the characterisation of stationary processes through \eqref{langevin} in dimension one, provided in \cite{Viitasaari-2016a}, can be generalised naturally to the multidimensional setting, and motivates the statistical analysis of Equation \eqref{langevin}. For this, we  present the following definition for the class $\mathcal{G}_{\bh}$ of possible drivers $G$.
\begin{defi}
\label{defi:GH}
Let $\bh > 0$ be fixed. Let $G= (G_t)_{t\in\mathbb{R}}$ be an $n$-dimensional stochastic process with stationary increments and $G_0=0$. We denote $G\in\mathcal{G}_\bh$ if

\begin{equation*}
\lim_{u\to\infty}\int_{-u}^0 e^{\bh s}dG_s
\end{equation*}
defines as an almost surely finite random vector.
\end{defi}
\begin{rem}
In the one-dimensional setting, existence of certain logarithmic moments are sufficient to ensure $G \in \mathcal{G}_\bh$. This result can be extended to the multidimensional setting in a straightforward manner. Consequently, our estimation procedure, that does rely on the existence of the second moments, always guarantees $G \in \mathcal{G}_\bh$.
\end{rem}
\begin{theorem}
\label{main1}
Let $\bh>0$ be fixed. A continuous time $n$-dimensional stochastic process $U = (U_t)_{t\in\mathbb{R}}$ is stationary if and only if it is the unique solution of the Langevin equation \eqref{langevin} for some $G\in\mathcal{G}_\bh$ and the initial value
\begin{equation}
\label{initial2}
U_0 = \int_{-\infty}^0 e^{\bh s} dG_s.
\end{equation}
That is
\begin{equation}
\label{stationarysolution}
U_t = e^{-\bh t}\int_{-\infty}^t e^{\bh s} dG_s.
\end{equation}
Moreover, the process $G$ is unique.
\end{theorem}

Motivated by this result, let us now turn our attention to the statistical analysis of \eqref{langevin}. That is, we suppose that the solution $U$ is observed, and our aim is to define an estimator for the unknown parameter $\bh$. Our approach is based on utilizing the cross-covariance matrices, and for this reason we require some moment assumptions. In the sequel, we assume that the components $G^{(i)}$ of the driver $G$ satisfy, for all $i=1,\ldots,n$,
\begin{equation}
\label{eq:var-bound}
\sup_{s\in[0,1]} \e \left[G^{(i)}_s\right]^2 < \infty.
\end{equation}
This assumption ensures that $G$ is square-integrable, and consequently so is the solution $U$. We remark that this assumption ensures that $G \in \mathcal{G}_\bh$ (cf. Lemma \ref{lemma:existenceofintegral}). Also, without loss of generality, we assume that $G$ is centred, i.e. $\e(G_t) = 0$ for every $t\in\mathbb{R}$. This gives that also $\e(U_t)=0$ for every $t\in\mathbb{R}$.

Let us now introduce some notation. With $\gammah(t)$ we denote the cross-covariance matrix of the stationary solution $U$, namely

\begin{equation}
\label{cross-covariance}
\begin{split}
\gammah(t) = \e(U_tU_0^\top) &= \begin{bmatrix}
\e(U_t^{(1)} U_0^{(1)}) & \e(U_t^{(1)} U_0^{(2)})& \hdots& \e(U_t^{(1)} U_0^{(n)})\\
\e(U_t^{(2)} U_0^{(1)}) & \e(U_t^{(2)} U_0^{(2)})& \hdots& \e(U_t^{(2)} U_0^{(n)})\\
\vdots&\vdots&\ddots&\vdots\\
\e(U_t^{(n)} U_0^{(1)}) & \e(U_t^{(n)} U_0^{(2)})&\hdots& \e(U_t^{(n)} U_0^{(n)})
\end{bmatrix}.
\end{split}
\end{equation}
Notice that $\gammah(-t)= \gammah(t)^\top$. In addition, we denote a single element $\e(U_t^{(i)} U_0^{(j)})$ by $\gamma_{i,j}(t)$. We also define the following matrix coefficients for every $t\geq0$.

\begin{align}
\Bh_t &= \int_0^t \left( \gammah(s) - \gammah(s)^\top \right) ds \label{coefofriccati1}\\
\Ch_t &= \int_0^t \int_0^t \gammah(s-u) du ds \label{coefofriccati2}\\
\Dh_t &= \mathrm{cov}(G_t) - \mathrm{cov}(U_t-U_0). \label{coefofriccati3}
\end{align}

\begin{rem}
\label{rem:covariance}
The cross-covariance $\gammah(t)$ can be computed explicitly from \eqref{stationarysolution}. For representations in the case when $G$ has independent components, see Lemma \ref{lma:crosscov-rep-1} and Lemma \ref{lemma:covofU2}.
\end{rem}
With the help of the above notation, we are able to write the parameter matrix $\bh$ as a solution to the so-called continuous-time algebraic Riccati equation (CARE), with matrices $\Bh_t$, $\Ch_t$, and $\Dh_t$ as coefficients. This will  lead to a natural estimator for $\bh$.
\begin{theorem}
\label{main2}
Let $U$ be the solution of the Langevin equation \eqref{langevin} with $\bh >0$ and initial \eqref{initial2}. Then, for every $t\geq 0$, the CARE 
\begin{equation}
\label{riccati}
\Bh_t^\top \bh + \bh \Bh_t - \bh \Ch_t \bh + \Dh_t =0
\end{equation}
is satisfied.
\end{theorem}

\begin{rem}
\label{rem:1d}
In the one-dimensional case $\Bh_t \equiv 0$. After a change of variable \eqref{riccati} transforms into
\begin{equation*}
2\bh^2\int_0^t \gamma(z) (t-z)dz = v(t) + 2\gamma(t) - 2\gamma(0),
\end{equation*}
where $v(t)$ is the variance function of $G$. From this, we can compute a solution $\bh>0$ easily whenever $\int_0^t \gamma(z)(t-z)dz \neq 0$ and $v(t)$ is known.  More generally, the coefficients $\Bh_t$, $\Ch_t$, and $\Dh_t$ can be computed from the observed process $U$ if one value of the covariance matrix function $t \mapsto \mathrm{cov}(G_t)$ of the noise is known. In the literature, it is a typical assumption that the variance function of the noise is known completely (up to scaling).
\end{rem}
From practical point of view, it is desirable that \eqref{riccati} admits a unique positive definite solution. Indeed, then the solution is automatically the correct parameter matrix $\bh$. Moreover, uniqueness of the solution is also a wanted feature for numerical \\methods. If the coefficient matrices $\Ch_t$ and $\Dh_t$ are both positive definite, then the solution is unique (in the set of positive semidefinite matrices). In our model, it turns out that this is usually the case if one chooses $t$ appropriately. A detailed discussion on the matter is postponed to Subsection \ref{subsec:uniqueness} (See also Remark \ref{rem:simulation-and-practice} below on how $t$ can be chosen in practice).
\begin{rem}
Even if the solution is unique, \eqref{riccati} is rarely solvable in a closed form. Thus, in practice or for simulations, one has to apply some numerical method. On the other hand, even in the one-dimensional general Gaussian setup one may need to rely on numerical approximations. For example, the ergodicity estimator studied in \cite{Sottinen-Viitasaari-2017a} is based on a function $\psi^{-1}$ that can be computed explicitly only in some particular cases. Actually, applying our method to the one-dimensional case we observe a closed form expression for the solution (cf. Remark \ref{rem:1d}). For numerical methods associated to \eqref{riccati}, see e.g. \cite{byers1987solving,laub1979schur} and the monograph \cite{bini2012numerical}.
\end{rem}
In the sequel, we assume that $t$ is chosen such that $\Ch_t,\Dh_t>0$, guaranteeing that $\bh$ is the unique solution to \eqref{riccati}. For notational simplicity, we will omit the subindex $t$ and simply write
\begin{equation}
\label{CARE}
\Bh^\top \bh + \bh \Bh - \bh \Ch \bh + \Dh = 0
\end{equation}
whenever confusion cannot arise.

Suppose now that we have an observation window $[0,T]$. We define estimators $\hat{\Bh}_T$, $\hat{\Ch}_T$, and $\hat{\Dh}_T$ for the coefficient matrices by replacing $\gamma_{i,j}(s)$ with any cross-covariance estimator $\hat{\gamma}_{T,i,j}(s)$ in the defining equations \eqref{coefofriccati1}-\eqref{coefofriccati3} (see Section \ref{subsection:gaussian} for an example of covariance estimator). We also write $\Delta_T \Bh=\hat{\Bh}_T -\Bh$,
$\Delta_T \Ch=\hat{\Ch}_T -\Ch$, and $\Delta_T \Dh=\hat{\Dh}_T -\Dh$. This leads to a perturbed CARE that gives us an estimator for $\bh$.
\begin{defi}
\label{def:estimator}
The estimator $\hat{\bh}_T$ is defined as the positive semidefinite solution to the perturbed CARE
\begin{equation}
\label{perCARE}
\hat{\Bh}_T^\top\hat{\bh}_T + \hat{\bh}_T\hat{\Bh}_T - \hat{\bh}_T\hat{\Ch}_T\hat{\bh}_T + \hat{\Dh}_T = 0
\end{equation}
whenever there exists a unique solution in the class of positive semidefinite matrices. If the solution does not exists, we set $\hat{\bh}_T = 0$.
\end{defi}
The idea of our estimator is that if the estimators $\hat{\gamma}_{T,i,j}(s)$ are  consistent and $\Ch,\Dh>0$ in the original CARE \eqref{riccati}, then the perturbed version \eqref{perCARE} automatically has a unique solution $\hat{\bh}_T$ (with probability increasing to one as $T$ grows), that converges strongly to $\bh$. 
\begin{theorem}
\label{theo:pertur}
Suppose $\Ch, \Dh >0$ and assume that
\begin{equation}
\label{covarianceinP}
\sup_{s\in[0,t]} \Vert\hat{\gamma}_{T}(s)- \gamma(s)\Vert \overset{\p}{\longrightarrow} 0.
\end{equation}
Then for $\hat{\bh}_T,$ given by Definition \ref{def:estimator}, we have
\begin{equation}
\label{solutioninP}
\Vert\hat{\bh}_T - \bh\Vert \overset{\p}{\longrightarrow} 0.
\end{equation}
\end{theorem}
\begin{rem}
\label{rem:sufficient}
If the convergence in \eqref{covarianceinP} holds almost surely, we obtain a strong consistent estimator, i.e.
\begin{equation*}
\Vert\hat{\bh}_T - \bh\Vert  \overset{\text{a.s.}}{\longrightarrow} 0.
\end{equation*}
Moreover, by our proof (cf. Lemma \ref{lemma:frobenius}) we obtain that instead of \eqref{covarianceinP}, weaker conditions
$$
\int_0^t \Vert \hat{\gamma}_T(s) - \gamma(s)\Vert ds \overset{\p}{\longrightarrow} 0
$$ 
and
$$
\Vert \hat{\gamma}_T(\tau) - \gamma(\tau)\Vert \overset{\p}{\longrightarrow} 0, \tau \in \{0,t\}
$$
are sufficient. These conditions are usually easier to verify in practice.
\end{rem}
\begin{rem}
\label{rem:simulation-and-practice}
In practice, one does not know the underlying exact model, and thus one cannot determine whether for given $t$ we have $\Ch_t,\Dh_t>0$. However, one can always pre-check whether, for a given $t$, matrices $\hat{\Ch}_T, \hat{\Dh}_T$ that are computed from the observations are positive definite. This together with \eqref{covarianceinP} indicates $\Ch_t,\Dh_t> 0$ implying that the original CARE \eqref{CARE} has a unique positive semidefinite solution $\bh$ (cf. Theorem \ref{theorem:pertur}). Now Theorem \ref{theo:pertur} applies, and consequently one can estimate $\bh$ from the observations by applying any numerical method for CARE, without pre-knowledge on positive definitiness of $\Ch_t$ and $\Dh_t$. This practical approach can also be used for simulations. 
\end{rem}
By Theorem \ref{theo:pertur}, the consistency of $\hat{\bh}_T$ is inherited from the consistency of $\hat{\gamma}_{T}$. Similarly, the rate of convergence and the limiting distribution for $\hat{\bh}_T$ follow from the convergence rate and the limiting distribution of $\hat{\gamma}_{T}$, respectively.
\begin{theorem}
\label{theo:fclt}
Let $X = (X_s)_{s\in[0,t]}$ be an $n^2$-dimensional stochastic process with continuous paths almost surely and let $l(T)$ be an arbitrary rate function. If
\begin{equation}
\label{functionalcovariance}
l(T)\vect (\hat{\gamma}_T(s)- \gamma(s)) \overset{\text{law}}{\longrightarrow} X_s
\end{equation}
in the uniform topology of continuous functions, then: \begin{itemize}
\item[(1)] If $\tilde{X}_s$ is the permutation of elements of $X_s$ that corresponds to the order of elements of $\vect(\gamma(s)^\top)$, we have
\begin{equation*}
l(T) \vect(\Delta_T \Ch, \Delta_T \Bh, \Delta_T \Dh) \overset{\text{law}}{\longrightarrow}  \begin{bmatrix*}
\int_0^ t (t-s)(X_s + \tilde{X}_s) ds\\
\int_0^t \left(X_s -\tilde{X}_s\right) ds\\
2X_0 - X_t -\tilde{X}_t
\end{bmatrix*}
\eqqcolon L_1(X).
\end{equation*}
\item[(2)] If $\Ch, \Dh > 0$ and $\hat{\bh}_T$ is given by Definition \ref{def:estimator}, we have
\begin{equation*}
l(T)  \vect(\hat{\bh}_T - \bh) \overset{\text{law}}{\longrightarrow} L_2(L_1(X)),
\end{equation*}
where $L_2: \mathbb{R}^{3n^2} \rightarrow \mathbb{R}^{n^2}$ is a linear operator depending only on $\bh$, $t$ and the cross-covariance of $G$.
\end{itemize}
\end{theorem}
\begin{rem}
The operator $L_2$ is given explicitly in the proof, see page 25.
\end{rem}

\subsection{On the uniqueness of the solution to \eqref{riccati}}
\label{subsec:uniqueness}
The uniqueness of the solution to \eqref{riccati} is crucially important, as otherwise we cannot guarantee that a convergent numerical scheme (which we have to apply in practice) converges to the true parameter $\bh$. In our case, it turns out that one can usually choose $t$ such that $\Ch_t,\Dh_t >0$ giving us  uniqueness. We next address the uniqueness issue particularly in our case. For the general theory of algebraic Riccati equations, see e.g. \cite{lancaster1995algebraic}.

We begin with some definitions.
\begin{defi}
\label{defi:stable}
A square matrix $A$ is stable if all its eigenvalues are in the open left half-plane.
\end{defi}

\begin{defi}
\label{defi:stabilizable}
A matrix pair $(A, B)$ is stabilizable if there exists a matrix $K$ such that $A + BK$ is stable.
\end{defi}

\begin{defi}
\label{defi:detectable}
A real matrix pair $(A, B)$ is detectable if $(B^\top, A^\top)$ is stabilizable.
\end{defi}
We utilise the following uniqueness result (for more details on the topic, see e.g. \cite{kucera1972contribution} or \cite{wonham1968matrix}) to our case under the assumption that $\Ch_t,\Dh_t\geq 0$.

\begin{lemma}
\label{lemma:uniqueness}
Let $\Ch_t, \Dh_t \geq 0$. If $(\Bh_t, \Ch_t)$ is stabilizable and $(\Dh_t, \Bh_t)$ is detectable, then the continuous time algebraic Riccati equation \eqref{riccati} has a unique positive semidefinite solution $\bh$. Furthermore, the matrix $\Bh_t - \Ch_t\bh$ is stable.
\end{lemma}
With this we obtain the following useful corollary.
\begin{kor}
Let $C_t,D_t>0$. Then \eqref{riccati} has a unique positive definite solution $\bh$.
\end{kor}
\begin{proof}
Let $S$ be any stable matrix and set $K_1 = \Ch_t^{-1}(S -\Bh_t)$ and $K_2 = (\Dh_t^\top)^{-1}(S-\Bh_t^\top)$. Then $\Bh_t + \Ch_tK_1 = S  = \Bh_t^\top + \Dh_t^\top K_2$, and the conditions of Lemma \ref{lemma:uniqueness} are satisfied. Thus CARE \eqref{riccati} has a unique solution $\bh\geq 0$, which is then automatically the true parameter matrix $\bh>0$.
\end{proof}
Let us now address when one can choose $t$ such that $\Ch_t,\Dh_t>0$. For this recall that
\begin{equation*}
\Ch_t = \e \left[\int_0^t U_s ds \left( \int_0^t U_s ds\right)^\top \right] = \mathrm{cov}\left(\int_0^t U_s ds \right)
\end{equation*}
and
$$
\Dh_t = \mathrm{cov}(G_t) - \mathrm{cov}(U_t-U_0).
$$
Thus $\Ch_t\geq 0$ for every $t$. Consider now the matrix $\Dh_t$. By stationarity of $U$ the elements of $\mathrm{cov}(U_t-U_0)$ are uniformly bounded, implying
$a^\top \mathrm{cov}(U_t-U_0) a < C\Vert a\Vert^2$ for some constant $C$. On the other hand, we have
$a^\top \mathrm{cov}(G_t) a \geq \lambda_{min}\Vert a\Vert^2$, where $\lambda_{min}$ is the smallest eigenvalue of $\mathrm{cov}(G_t)$.
Thus
\begin{equation*}
a^\top \Dh_t a \geq (\lambda_{min} - C)\Vert a\Vert^2,
\end{equation*}
implying $\Dh_t>0$ provided that $\lambda_{min}$ grows sufficiently. This happens, for example, when $G$ has independent components with growing variances.

Consider next the matrix $\Ch_t$. Since $\Ch_t\geq 0$ always, it suffices to find one $t$ such that $\Ch_t>0$. Let us, for a moment, suppose that this is not possible. Then $\rank(\Ch_t) \leq n-1$ implying that there exists a (vector-valued) function $a(t)$ such that, almost surely and for all $t$,
\begin{equation*}
a(t)^\top \int_0^ t U_s ds = 0.
\end{equation*}
Without loss of generality we can assume that $a(t)$ is normalised and oriented consistently. Furthermore, it follows from the continuity of $\int_0^t U_sds$ that $a(t)^\top$ is also continuous. This further implies that $a(t)^\top \int_0^ t U_s ds$ is indistinguishable from the zero process meaning that there exists $B\subset \Omega$ such that $\p(B) = 1$ and
\begin{equation}
\label{indish}
\int_0^ t U_s(\omega) ds \in M_t^{n-1} \quad\text{for every }\omega\in B \text{ and } t\in\mathbb{R_+},
\end{equation}
where $M_t^ {n-1}$ is a $n$-$1$-dimensional subspace of $\mathbb{R}^ n$. We claim that this implies also degeneracy of the process $U$ itself. We, again, proceed by contradiction and assume that there exists $\omega_i \in B$, $i=1,2,\hdots,n$ such that the vectors $U_0(\omega_i)$ are linearly independent. Then the matrix
\begin{equation}
\label{eq:U_0-invert}
\left[ U_0(\omega_1), \cdots, U_0(\omega_n)\right]
\end{equation}
is invertible. On the other hand, for any $\varepsilon>0$ we can apply the mean value theorem to find $\delta>0$ such that
\begin{equation*}
\int_0^\delta U_s(\omega_i) ds =  (U_0(\omega_i) + \varepsilon_{\omega_i, \delta})\delta, \quad\text{with } \Vert\varepsilon_{\omega_i, \delta}\Vert < \varepsilon.
\end{equation*}
Thus, by continuity of the eigenvalues and invertibility of the matrix \eqref{eq:U_0-invert}, the matrix
\begin{equation*}
\left[ U_0(\omega_1) + \varepsilon_{\omega_1, \delta} \cdots U_0(\omega_n) + \varepsilon_{\omega_n, \delta}\right]
\end{equation*}
is invertible as well provided that $\varepsilon$ is chosen small enough. This contradicts \eqref{indish}, meaning that if $\rank(\Ch_t) \leq n-1$, then we have \eqref{indish} and $U_0(\omega) \in \tilde{M}_0^{n-1} \text{ for all } \omega \in B$ as well. Now stationarity of $U$ implies that $\p(U_t \in \tilde{M}_0^{n-1}) = 1$ for all $t\in\mathbb{R}$, meaning that $U$ is a degenerate process. In particular, then
\begin{equation*}
b^\top U_0 = \int_{-\infty}^ 0 b^\top e^{\bh s} dG_s  = \int_{-\infty}^ 0 \sum_{i=1}^ n \left(b^ \top e^{\bh s} \right)^{(i)} dG_s^ {(i)} = \sum_{i=1}^ n \int_{-\infty}^ 0  \left(b^ \top e^{\bh s} \right)^{(i)} dG_s^ {(i)} \overset{\text{a.s.}}{=} 0
\end{equation*}
for some non-zero vector $b$. If now $G$ has independent components, then we would also get
\begin{equation*}
\int_{-\infty}^ 0  \left(b^ \top e^{\bh s} \right)^{(i)} dG_s^ {(i)} \overset{\text{a.s.}}{=} 0\quad\text{for all } i.
\end{equation*}
For many interesting processes $G^{(i)}$ this would further imply $\left(b^ \top e^{\bh s} \right)^{(i)} \equiv 0 $ leading to a contradiction since $e^{\bh s}$ is of full-rank. In particular, this is the case whenever $G^{(i)}$ is a Gaussian process for which Wiener integral is injective (for details on Wiener integrals, see e.g. \cite{Janson-1997}). Such Gaussian processes include, among others, Brownian motions and fractional Brownian motions. Finally, we note that in general, if we have a set of observations $\{ U_t(\omega)\}_{t\in I}$ (with a fixed $\omega$) and $\mathrm{span}\{  U_t(\omega)\}_{t\in I} = \mathbb{R}^n$, then one can always find $t$ such that $\Ch_t>0$.

\subsection{Application to Gaussian processes}
\label{subsection:gaussian}
In this subsection, we illustrate the applicability of our results to the Gaussian case. That is, we suppose that the components $G$ are independent Gaussian processes $G^{(i)}$. We state the results under conditions on the cross-covariance $\gamma(t)$. In practice, one can verify the assumptions for a given model by computing $\gamma(t)$ from the variance matrix $v(t) = \e [G_t G_t^\top]$. In particular, different representations for $\gamma(t)$ are given in Subsection \ref{subsec:proof-4}. We apply these representations to prove that all our results are applicable, whenever the components $G^{(i)}$ are independent fractional Brownian motions with Hurst indices $H_i < \frac34$ (cf. Corollary \ref{kor:fbm-clt} below).

We first note that, by assumption, the components $G^{(i)}$ have continuous paths almost surely. By Gaussianity, this implies $L^2$ continuity as well, and hence \eqref{eq:var-bound} is valid, giving $G \in \mathcal{G}_{\bh}$. For the cross-covariance estimator $\hat{\gamma}$, we use standard
$$
\hat{\gamma}_{T}(\tau) = \frac{1}{T}\int_0^T U_{s+\tau} U_s^\top ds.
$$
The following result gives us the consistency immediately, and covers all ergodic systems.
\begin{prop}
\label{prop:Gaussian-consistency}
Let $G$ be a vector of Gaussian processes. If 
$
\lim_{t\to\infty}\Vert \gamma(t)\Vert = 0,
$
then 
$
\Vert \hat{\bh}_T - \hat{\bh}\Vert \overset{\p}{\longrightarrow} 0.
$
\end{prop}
Proposition \ref{prop:Gaussian-consistency} guarantees that we can apply Theorem \ref{theo:pertur} if the cross-covariance $\gamma(t)$ vanishes at infinity. Similarly, we may apply Theorem \ref{theo:fclt} if $\gamma(t)$ decays rapidly enough.
\begin{theorem}
\label{theorem:Gaussian-clt}
Suppose that $\gamma(r)$ is differentiable for almost all $r$ and 
$$
\max\left(\Vert\gamma'(r)\Vert,\Vert \gamma(r)\Vert\right) \leq h(r)
$$
for some non-increasing function $h(r)$ such that, for some $K>0$, we have $h(r) \in L^1([0,K])$ and $h(r) \in L^2([K,\infty))$.
Then 
\begin{equation}
\sqrt{T}\vect (\hat{\gamma}_T(s)- \gamma(s)) \overset{\text{law}}{\longrightarrow} X_s
\end{equation}
in the uniform topology of continuous functions, where $X$ is an $n^2$-dimensional centered Gaussian process.  In particular, Theorem \ref{theo:fclt} is applicable.
\end{theorem}
\begin{rem}
The cross-covariance $\e [X_\tau X_\eta^\top]$ of the process $X$ can be computed explicitly, and it consists of elements 
\begin{equation}
\label{eq:cross-cov-X}
\int_0^\infty \gamma_{i,j}(r+\tau)\gamma_{p,q}(r+\eta)dr, \quad i,j,p,q \in \{1,2,\ldots,n\}
\end{equation}
in the order corresponding to $\vect (\hat{\gamma}_T(s)- \gamma(s))$. We also note that assumptions on the function $h$ ensures that the terms \eqref{eq:cross-cov-X} are finite. 
\end{rem}
\begin{rem}
By representation \eqref{covofU}, the differentiability of $\gamma(r)$ follows provided that the variance functions $v_i(r)$ of the components $G^{(i)}$ are differentiable. 
\end{rem}
\begin{rem}
Convergence of finite dimensional distributions in the above result follows from some well-known facts. However, to the best of our knowledge, tightness of $\hat{\gamma}_T(t)$ with lag $t$ as a free parameter, has not previously been acknowledged in the literature making it the most significant point of our example.
\end{rem}
To end this section we apply Theorem \ref{theorem:Gaussian-clt} to the case of \emph{multidimensional fractional Ornstein-Uhlenbeck processes}. Recall that a fractional Brownian motion $B^H$ with Hurst index $H\in(0,1)$ is a centered Gaussian process with covariance
$$
R_{B^H}(t,s) = \frac{1}{2}\left[t^{2H} + s^{2H} - |t-s|^{2H}\right].
$$
\begin{kor}
\label{kor:fbm-clt}
Let $G$ be a vector of independent fractional Brownian motions $B^{H_i}$ with Hurst indices $H_i<\frac34$. Then 
\begin{equation}
\sqrt{T}\vect (\hat{\gamma}_T(s)- \gamma(s)) \overset{\text{law}}{\longrightarrow} X_s
\end{equation}
in the uniform topology of continuous functions, where $X$ is an $n^2$-dimensional centered Gaussian process. In particular, Theorem \ref{theo:fclt} is applicable.
\end{kor}
\begin{rem}
\label{rem:tightness}
By carefully examining our proof we actually observe that the tightness holds for arbitrary values of the Hurst indices $H_i \in (0,1)$. Indeed, this follows since $\gamma(t) \sim t^{2\h_{max}-2}$ at infinity, giving us the expected rate function $l(T)$ (cf. Proposition \ref{prop:tightness}). Thus it suffices to study only the convergence of finite dimensional distributions.
\end{rem}
The above results are obviously just illustrations how our general theorems can be applied. For example, it is straightforward to check the applicability of Theorem \ref{theorem:Gaussian-clt} in the multidimensional versions of \emph{the fractional Ornstein-Uhlenbeck process of the second kind} or \emph{the bifractional Ornstein-Uhlenbeck process of the second kind} (see \cite{Sottinen-Viitasaari-2017a} and the references therein for definitions). Indeed, it can be shown that, as in the univariate case, covariances $\gamma_{ij}(t)$ decay exponentially. Similarly, in the case of multidimensional fractional Ornstein-Uhlenbeck process where some of the Hurst indices $H_i$ satisfy $H_i\geq \frac34$, we can obtain a limiting object, but with different rate and possibly different limiting object $X_s$. For example, if $\max H_i = \frac34$, then the rate is $\frac{\sqrt{T}}{\sqrt{\log T}}$ instead of standard $\sqrt{T}$, while the limiting process $X_s$ is still Gaussian. If $\max H_i>\frac34$, one expects to have Rosenblatt components in $X$. Indeed, this is a well-known fact in dimension one (see, e.g. \cite{david-general-H}), and the tightness holds on the full range $H_i \in (0,1)$ (see Remark \ref{rem:tightness}).

\section{Proofs}
\label{sec:proofs}
For the reader's convenience, we divide this section into six subsections. The first subsection, Subsection \ref{subsec:proof-1}, provides a proof of Theorem \ref{main1} motivating our model. Subsection \ref{subsec:proof-2} contains a proof of Theorem \ref{main2} that leads to the definition of our estimator $\hat{\bh}_T$. In Subsection \ref{subsec:proof-3} we prove our results, Theorem \ref{theo:pertur} and Theorem \ref{theo:fclt}, concerning the asymptotic behaviour of $\hat{\bh}_T$. In Subsection \ref{subsec:proof-4}, we provide representations for the cross-covariance $\gamma(t)$. These representations will then be applied in Subsection \ref{subsec:proof-5}, where we prove results related to our Gaussian example.
\subsection{Proof of Theorem \ref{main1}}
\label{subsec:proof-1}
The proof of Theorem \ref{main1} follows the strategy of \cite{Viitasaari-2016a}. However, in multidimensional setting one has to be carefully, e.g. whether matrices commute or not. In addition, we need to extend concepts such as self-similarity to the matrix-valued case. For this reason, we do not omit the proof even though it is partly very similar to the univariate case.

We begin with the following definition of $\bh$-self-similar processes, where $\bh$ is a matrix.
\begin{defi}
Let $\bh>0$. An $n$-dimensional stochastic process $X = (X_t)_{t \geq 0}$ with $X_0 = 0$ is $\bh$-self-similar if
\begin{equation*}
(X_{at})_{t\geq 0} \overset{law}{=} a^\bh(X_t)_{t\geq 0}
\end{equation*}
for every $a >0$ in the sense of finite dimensional distributions. Here the matrix exponent is defined through the matrix exponential function $a^\bh = e^{\bh\log a}$.
\end{defi}
The following remark illustrates the necessity of positive definiteness of $\bh$.

\begin{rem}
\label{rem:dimensionreduction}
The assumption $\bh>0$ is natural, as otherwise we may reduce the number of dimensions. Indeed, if $\bh\geq 0$ with one eigenvalue $\lambda_1=0$, then the eigendecomposition $\bh = Q\Lambda Q^\top$ gives
\begin{equation*}
X_a \overset{\text{law}}{=} Q e^{\Lambda \log a}Q^\top X_1 = Q \mathrm{diag} (e^{\lambda_i \log a})Z,
\end{equation*}
with $Z = Q^\top X_1$. Since $Q$ is orthogonal, it follows that
\begin{equation*}
\Vert X_{a}\Vert \overset{\text{law}}{=}\Vert Q \mathrm{diag} (e^{\lambda_i \log a})Z\Vert \geq |Z^{(1)}| = |(Q^\top X_1)^{(1)} |.
\end{equation*}
In particular, using continuity of $X$ and letting $a\to 0$ yields $(Q^\top X_1)^{(1)} = 0$. This means that $X$ is an $(n-1)$-dimensional process. Similarly, if $G\in\mathcal{G}_\bh$ with $\bh$ having zero as an eigenvalue with algebraic multiplicity equal $k$, then $G$ degenerates to an $(n-k)$-dimensional process.
\end{rem}
\begin{defi}
Let $\bh > 0$. In addition, let $U = (U_t)_{t\in\mathbb{R}}$ and $X = (X_t)_{t \geq 0}$ be $n$-dimensional stochastic processes. We define
\begin{align*}
(\mathcal{L}_\bh U)_t &= t^\bh U_{\log t}, \quad \text{for } t>0\\
(\mathcal{L}^{-1}_\bh X)_t &= e^{-\bh t} X_{e^t}, \quad \text{for } t\in\mathbb{R}.
\end{align*}
\end{defi}
The following result extends the one-to-one correspondence between $\bh$-self-similar processes and stationary processes to the matrix-valued case. We use the name Lamperti transform for our matrix-valued version in honour to the original univariate result.
\begin{theorem}[Lamperti]
\label{lamperti}
Let $\bh >0$. Let $(U_t)_{t\in\mathbb{R}}$ be an $n$-dimensional stationary process. Then $( \mathcal{L}_\bh U)_t$ is $\bh$-self-similar. Conversely, let $(X_t)_{t\geq 0}$ be an $n$-dimensional $\bh$-self-similar process. Then $(\mathcal{L}^{-1}_\bh X)_t$ is stationary.
\begin{proof}
Suppose first that $(U_t)_{t\in\mathbb{R}}$ is stationary. Define $Y_t = ( \mathcal{L}_\bh U)_t = t^\bh U_{\log t}$. Let $a >0$ and $[t_1, t_2, \cdots, t_m]^\top \in \mathbb{R}_+^m$. Then

\begin{equation*}
\begin{split}
(Y_{at_1}, Y_{at_2}, \cdots, Y_{at_m}) &= (a^\bh t_1^\bh U_{\log at_1}, a^\bh t_2^\bh U_{\log at_2}, \cdots, a^\bh t_m^\bh U_{\log at_m})\\
&= (a^\bh t_1^\bh U_{\log a + \log t_1}, a^\bh t_2^\bh U_{\log a+ \log t_2}, \cdots, a^\bh t_m^\bh U_{\log a + \log t_m})\\
&\overset{law}{=} (a^\bh Y_{t_1}, a^\bh Y_{t_2},\cdots, a^\bh Y_{t_m}).
\end{split}
\end{equation*}
Now let $\bh = Q\Lambda Q^\top$ be an eigendecomposition of $\bh$. Then

\begin{equation}
\label{matrixexponent}
t^\bh = e^{\bh \log t} = Q \sum_{k=0}^ \infty \frac{\Lambda^ k (\log t)^ k}{k!} Q^\top= Q e^{\Lambda \log t} Q^\top,
\end{equation}
where $e^{\Lambda \log t}$ is a diagonal matrix with diagonal elements of the form $e^{\lambda_i \log t}$. Since $\lambda_i > 0$ for every $i=1,2,...,n$, we notice that $\lim_{t\to 0} Y_t = 0$ in probability, and hence, $Y$ is $\bh$-self-similar.\\
Next, suppose that $(X_t)_{t \geq 0}$ is $\bh$-self-similar. Define $Y_t = (\mathcal{L}^{-1}_\bh X)_t = e^{-\bh t} X_{e^t}$. Let $s\in\mathbb{R}$ and $[t_1, t_2, \cdots, t_m]^\top \in \mathbb{R}^m$. Then

\begin{equation*}
\begin{split}
(Y_{t_1+s}, Y_{t_2+s}, \cdots, Y_{t_m+s}) &= (e^{-\bh(t_1+s)} X_{e^{t_1+s}}, e^{-\bh(t_2+s)} X_{e^{t_2+s}}, \cdots, e^{-\bh(t_m+s)} X_{e^{t_m+s}})\\
&\overset{law}{=}  (e^{-\bh t_1} X_{e^{t_1}}, e^{-\bh t_2} X_{e^{t_2}}, \cdots, e^{-\bh t_m} X_{e^{t_m}})\\
&= (Y_{t_1}, Y_{t_2},\cdots, Y_{t_m})
\end{split}
\end{equation*}
concluding the proof.
\end{proof}
\end{theorem}
The following lemma is a straightforward extension of a similar univariate result of \cite{Viitasaari-2016a}. For the reader's convenience, we present the proof here.
\begin{lemma}
\label{similartoG}
Let $(X_t)_{t\geq 0}$ be an $n$-dimensional $\bh$-self-similar process. Define $Y= (Y_t)_{t\in\mathbb{R}}$ by
\begin{equation*}
Y_t = \int_0^t e^{-\bh u} dX_{e^u}.
\end{equation*}
Then $Y\in\mathcal{G}_\bh$.
\begin{proof}
Clearly $Y_0 = 0$. In addition

\begin{equation*}
\int_{-\infty}^0 e^{\bh u} dY_u = \int_{-\infty}^0 dX_{e^u} = X_1 - \lim_{t\to -\infty} X_{e^t}  \overset{\p}{=} X_1.
\end{equation*}
Now, let $t,s,h \in \mathbb{R}$. Then
\begin{equation*}
\begin{split}
Y_t - Y_s &= \int_s^t e^{-\bh u} dX_{e^u} = \int_{s+h}^{t+h} e^{-\bh(v-h)} dX_{e^{v-h}}\\
&\overset{law}{=} \int_{s+h}^{t+h} e^{-\bh v} dX_{e^v} = Y_{t+h} - Y_{s+h},
\end{split}
\end{equation*}
where we have used the change of variable $u = v-h$. The penultimate equation can be verified by approximating the integral with Riemann sums, using self-similarity, and passing to the limit. Similarly, for multidimensional distributions
\begin{equation*}
\begin{split}
\begin{bmatrix}
Y_{t_1}-Y_{s_1}\\
Y_{t_2}-Y_{s_2}\\
\vdots\\
Y_{t_m}-Y_{s_m}
\end{bmatrix}
&=
\begin{bmatrix}
 \int_{s_1}^{t_1} e^{-\bh u} dX_{e^u}\\
 \int_{s_2}^{t_2} e^{-\bh u} dX_{e^u}\\
\vdots\\
 \int_{s_m}^{t_m} e^{-\bh u} dX_{e^u}
\end{bmatrix}
=
\begin{bmatrix}
 \int_{s_1+h}^{t_1+h} e^{-\bh(v-h)} dX_{e^{v-h}}\\
 \int_{s_2+h}^{t_2+h} e^{-\bh(v-h)}  dX_{e^{v-h}}\\
\vdots\\
 \int_{s_m+h}^{t_m+h} e^{-\bh(v-h)}  dX_{e^{v-h}}
\end{bmatrix}\\
&\overset{\text{law}}{=}
\begin{bmatrix}
 \int_{s_1+h}^{t_1+h} e^{-\bh v} dX_{e^v}\\
 \int_{s_2+h}^{t_2+h} e^{-\bh v} dX_{e^v}\\
\vdots\\
 \int_{s_m+h}^{t_m+h} e^{-\bh v} dX_{e^v}
\end{bmatrix}
=
\begin{bmatrix}
Y_{t_1+h}-Y_{s_1+h}\\
Y_{t_2+h}-Y_{s_2+h}\\
\vdots\\
Y_{t_m+h}-Y_{s_m+h}
\end{bmatrix}.
\end{split}
\end{equation*}
\end{proof}
\end{lemma}
We split the proof of Theorem \ref{main1} into three lemmas. The first one gives us the stationary solution to \eqref{langevin}.
\begin{lemma}
\label{lemma:first}
Let $\bh>0$ and $G\in\mathcal{G}_\bh$. Then the unique solution to the Langevin equation \eqref{langevin} with the initial condition
\begin{equation*}
U_0 = \int_{-\infty}^0 e^{\bh s} dG_s
\end{equation*}
is given by
\begin{equation*}
U_t = e^{-\bh t} \int_{-\infty}^t e^{\bh s} dG_s.
\end{equation*}
The solution is stationary.

\begin{proof}
By integration by parts
\begin{equation*}
U_t = e^{-\bh t} \int_{-\infty}^t e^{\bh s} dG_s = G_t -e^{-\bh t}\bh \int_{-\infty}^t e^{\bh s} G_s ds,
\end{equation*}
giving
\begin{equation*}
\begin{split}
dU_t &= dG_t - d\left(e^{-\bh t} \bh\right)  \int_{-\infty}^t e^{\bh s} G_s ds - e^{-\bh t}\bh d\left( \int_{-\infty}^t e^{\bh s} G_s ds\right)\\
&=dG_t - \bh d\left(e^{-\bh t}\right)  \int_{-\infty}^t e^{\bh s} G_s ds - \bh e^{-\bh t} e^{\bh t} G_t dt.
\end{split}
\end{equation*}
Here we have used the fact that $e^{-\bh t}$ and $\bh$ commute. Now

\begin{equation*}
\begin{split}
d U_t&= dG_t - \bh d\left(e^{-\bh t}\right)  \int_{-\infty}^t e^{\bh s} G_s ds - \bh  G_t dt\\
&= dG_t + \left(\bh^2e^{-\bh t} dt\right)   \int_{-\infty}^t e^{\bh s} G_s ds - \bh  G_t dt\\
&= dG_t - \bh\left(G_t - \bh e^{-\bh t}\int_{-\infty}^t e^{\bh s} G_s ds\right) dt \\
&= dG_t - \bh U_t dt
\end{split}
\end{equation*}
completing the proof of the first assertion. To show stationarity, change of variable $u = s-t$ gives us

\begin{equation*}
\begin{split}
U_t &= e^{-\bh t} \int_{-\infty}^t e^{\bh s} dG_s = e^{-\bh t} \int_{-\infty}^0 e^{\bh(u+t)} dG_{u+t}\\
&= \int_{-\infty}^0 e^{\bh u} dG_{u+t} \overset{\text{law}}{=} \int_{-\infty}^0 e^{\bh u} dG_u = U_0,
\end{split}
\end{equation*}
where we have used that $G$ has stationary increments. Again, the penultimate equation can be verified by approximating the integral with finite Riemann sums, using stationarity of increments, and passing to the limit. Treating multidimensional distributions similarly concludes the proof.

%%%%%%%%%%%%%%%%%%
%\iffalse
%\begin{equation*}
%\begin{split}
%\begin{bmatrix}
%U_{t_1+s}\\
%U_{t_2+s}\\
%\vdots\\
%U_{t_m+s}
%\end{bmatrix}
%&=
%\begin{bmatrix}
%e^{-\pmb{H}(t_1+s)} \int_{-\infty}^{t_1+s} e^{\pmb{H}u} dG_u\\
%e^{-\pmb{H}(t_2+s)} \int_{-\infty}^{t_2+s} e^{\pmb{H}u} dG_u\\
%\vdots\\
%e^{-\pmb{H}(t_m+s)} \int_{-\infty}^{t_m+s} e^{\pmb{H}u} dG_u
%\end{bmatrix}
%=
%\begin{bmatrix}
%e^{-\pmb{H}t_1} \int_{-\infty}^{t_1} e^{\pmb{H}r} dG_{r+s}\\
%e^{-\pmb{H}t_2} \int_{-\infty}^{t_2} e^{\pmb{H}r} dG_{r+s}\\
%\vdots\\
%e^{-\pmb{H}t_m} \int_{-\infty}^{t_m} e^{\pmb{H}r} dG_{r+s}
%\end{bmatrix}\\
%&\overset{\text{law}}{=}
%\begin{bmatrix}
%e^{-\pmb{H}t_1} \int_{-\infty}^{t_1} e^{\pmb{H}r} dG_{r}\\
%e^{-\pmb{H}t_2} \int_{-\infty}^{t_2} e^{\pmb{H}r} dG_{r}\\
%\vdots\\
%e^{-\pmb{H}t_m} \int_{-\infty}^{t_m} e^{\pmb{H}r} dG_{r}
%\end{bmatrix}
%=
%\begin{bmatrix}
%U_{t_1}\\
%U_{t_2}\\
%\vdots\\
%U_{t_m}
%\end{bmatrix}.
%\end{split}
%\end{equation*}
%\fi
%%%%%%%%%%%%%%%%%

\end{proof}
\end{lemma}
The next result gives us the other direction, i.e. it shows that stationary processes solve Langevin equation.
\begin{lemma}
\label{lemma:second}
Let $\bh>0$ be fixed and let $U = (U_t)_{t\in\mathbb{R}}$ be stationary process with continuous paths. Then $U$ is the unique solution to the Langevin equation \eqref{langevin} for some $G\in\mathcal{G}_H$ and the initial condition
\begin{equation*}
U_0 = \int_{-\infty}^0 e^{\bh s} dG_s.
\end{equation*}
\begin{proof}
Assume that $(U_t)_{t\in\mathbb{R}}$ is stationary. Then by Theorem \ref{lamperti} there exists a $\bh$-self-similar $(X_t)_{t \geq 0}$ such that $U_t = (\mathcal{L}^{-1}_\bh X)_t = e^{-\bh t} X_{e^t}$. Consequently
\begin{equation*}
\begin{split}
d U_t = d(e^{-\bh t}) X_{e^t} + e^{-\bh t} dX_{e^t} &= -\bh e^{-\bh t} X_{e^t} dt + e^{-\bh t} d X_{e^t}\\
&= -\bh U_t dt + e^{-\bh t} d X_{e^t}.
\end{split}
\end{equation*}
Now, define $Y = (Y_t)_{t\in\mathbb{R}}$ as in Lemma \ref{similartoG}. Then $Y\in\mathcal{G}_\bh$ and $dY_t = e^{-\bh t} dX_{e^t}$ concluding the proof.
\end{proof}
\end{lemma}
Finally, the next lemma provides us with the uniqueness of the noise.

\begin{lemma}
\label{lemma:third}
Let $\bh > 0$ be fixed. Then a process $U = (U_t)_{t\in\mathbb{R}}$ satisfies the Langevin equation \eqref{langevin} with the initial
\begin{equation}
\label{initial}
U_0 = \int_{-\infty}^0 e^{\bh s} dG_s
\end{equation}
for one process $G \in\mathcal{G}_\bh$ at the most.
\begin{proof}
Suppose that $G, \tilde{G} \in\mathcal{G}_\bh$ yield the same solution $U$ of the Langevin equation with the initial \eqref{initial}. Then for every $t\in\mathbb{R}$

\begin{equation*}
e^{\bh t}U_t =  \int_{-\infty}^t e^{\bh u} dG_u = \int_{-\infty}^t e^{\bh u} d\tilde{G}_u.
\end{equation*}
Let $s<t$, then

\begin{equation*}
\int_s^t e^{\bh u} dG_u = \int_s^t e^{\bh u} d\tilde{G}_u.
\end{equation*}
Integration by parts gives

\begin{equation*}
\int_s^t e^{\bh u} dG_u = e^{\bh t} G_t - e^{\bh s} G_s - \bh \int_s^t e^{\bh u} G_u du
\end{equation*}
yielding

\begin{equation*}
e^{\bh t} (G_t - \tilde{G}_t) - e^{\bh s} (G_s - \tilde{G}_s) = \bh \int_s^t e^{\bh u} (G_u - \tilde{G}_u) du.
\end{equation*}
By denoting $h(t) = e^{\bh t}(G_t - \tilde{G}_t)$, we obtain

\begin{equation*}
h(t) - h(s) = \bh \int_s^t h(u) du
\end{equation*}
or equivalently

\begin{equation}
\label{differential}
dh(t) = \bh h(t) dt.
\end{equation}
The general solution to \eqref{differential} reads
\begin{equation*}
h(t) = Q e^{\Lambda t} C,
\end{equation*}
where the columns of $Q$ are equal to the eigenvectors of $\bh$ and $\Lambda$ is the corresponding eigenvalue diagonal matrix, and $C$ is a constant vector. The initial $G_0 = \tilde{G}_0$ gives $QC = 0$. Since $Q$ is invertible, we conclude that $C= 0$.
\end{proof}
\end{lemma}
The proof of Theorem \ref{main1} now follows directly.
\begin{proof}[Proof of Theorem \ref{main1}]
The existence of a unique stationary solution to the Langevin equation is in fact the statement of Lemma \ref{lemma:first}. Conversely, the fact that stationary processes solve Langevin equations is the statement of Lemma \ref{lemma:second}. Finally, Lemma \ref{lemma:third} gives us the uniqueness of the noise.
\end{proof}

\subsection{Proof of Theorem \ref{main2}}
\label{subsec:proof-2}
In order to prove Theorem \ref{main2}, we begin by showing that for $\bh>0$, the \eqref{eq:var-bound} implies $G\in\mathcal{G}_\bh$, i.e. we show that for $\bh>0$,
\begin{equation*}
\int_{-\infty}^0 e^{\bh s} G_s ds
\end{equation*}
defines an almost surely finite random variable. For this we begin with the following very elementary lemma.

\begin{lemma}
\label{lemma:boundforvariance}
Let $G = (G_s)_{s\in\mathbb{R}}$ be a $1$-dimensional centred process with stationary increments, $G_0 = 0$, and $\sup_{s\in[0,1]}\e G_s^2 < \infty$. Then
\begin{equation*}
\mathrm{var}(G_s) \leq C(s+1)^2
\end{equation*}
for every $s\geq 0$, where $C$ is some positive constant depending only on the process $G$.
\begin{proof}
Let $s\geq 0$. Writing
$
G_s = G_s - G_{\lfloor s \rfloor} + \sum_{k=1}^{\lfloor s \rfloor} \left(G_k - G_{k-1}\right),
$
where $\lfloor \cdot \rfloor$ is the standard floor-function, and using stationarity of the increments together with the Minkowski's inequality gives
$\sqrt{\e G_s^ 2} \leq (s+1) \sup_{s\in [0,1]} \sqrt{\mathrm{var} (G_s)}
$.
\end{proof}
\end{lemma}
\begin{lemma}
\label{lemma:existenceofintegral}
Let $\bh>0$ and let $G$ satisfy \eqref{eq:var-bound}. Then

\begin{equation}
\label{limit?}
\lim_{u\to\infty} \int_{-u}^0 e^{\bh s} G_s ds
\end{equation}
exists almost surely.

\begin{proof}
Let $\bh = Q \Lambda Q^\top$ be an eigendecomposition of $\bh$. Then $e^ {\bh s} = Qe^{\Lambda s}Q^\top$, where $e^{\Lambda s}$ is a diagonal matrix with diagonal entries of the form $e^{\lambda_i s}$. Thus $\Vert e^{\Lambda s}\Vert = e^{\lambda_{min}s}$, where $\lambda_{min}$ is the smallest eigenvalue of $\bh$. Moreover, by orthogonality of $Q$, we have $\Vert Q\Vert \Vert Q^\top\Vert = 1$. Thus
\begin{equation}
\label{norm}
\begin{split}
\Vert e^ {\bh s} G_s \Vert &\leq \Vert Q \Vert \Vert e^{\Lambda s}\Vert \Vert Q^\top\Vert \Vert G_s\Vert\\
&= e^{\lambda_{min} s} \Vert G_s\Vert\\
&\leq e^{\lambda_{min} s} \sqrt{n} \max_{i} |G_s^{(i)}|.
\end{split}
\end{equation}
On the other hand, Lemma \ref{lemma:boundforvariance} gives
\begin{equation*}
\p\left(\left| e^{\frac{\lambda_{min}}{2} s} G_s^{(i)} \right|  > \epsilon\right) \leq \frac{\mathrm{var}\left(e^{\frac{\lambda_{min}}{2} s} G_s^{(i)}\right)}{\epsilon^2} \leq \frac{C_i e^{\lambda_{min}s} (1+|s|)^ 2}{\epsilon^ 2}.
\end{equation*}
Thus Borel-Cantelli implies
$\left| e^{\frac{\lambda_{min}}{2} s} G_s^{(i)} \right|\to 0$ almost surely as $s \to -\infty$, which further implies

\begin{equation}
\label{max}
\lim_{s\to -\infty} \max_{i} |e^{\frac{\lambda_{min}}{2} s} G_s^{(i)}| = 0
\end{equation}
almost surely. Hence we observe
\begin{equation*}
\begin{split}
\int_{-\infty}^0 \Vert e^{\bh s} G_s\Vert ds &\leq \int_{-\infty}^0 C_n e^{\lambda_{min}s} \max_i |G_s^{(i)}| ds\\
&= C_n \int_{-\infty}^ 0 e^{\frac{\lambda_{min}}{2}s} \max_i  |e^{\frac{\lambda_{min}}{2} s} G_s^{(i)}| ds\\
&\leq C_n \sup_{s\in(-\infty, 0]} \left \{\max_i  |e^{\frac{\lambda_{min}}{2} s} G_s^{(i)}|\right \} \int_{-\infty}^0 e^{\frac{\lambda_{min}}{2}s} ds,
\end{split}
\end{equation*}
where the supremum term is finite almost surely by \eqref{max}. This concludes the proof.
\end{proof}
\end{lemma}
We are now ready to prove Theorem \ref{main2}.

\begin{proof}[Proof of Theorem \ref{main2}]
Lemma \ref{lemma:existenceofintegral} together with assumption \eqref{eq:var-bound} gives us $G\in\mathcal{G}_\bh$, and by Theorem \ref{main1} the solution with initial \eqref{initial2} $U$ is stationary. Now \eqref{langevin} and $G_0=0$ gives, for every $t\geq 0$, that
\begin{equation}
\label{eq:G}
G_t- G_0 = G_t = U_t -U_0 + \bh\int_0^t U_s ds.
\end{equation}
In the sequel, we use short notation $
\Delta_t U = U_t-U_0$. Noticing that $\bh^\top = \bh,$ we now get from \eqref{eq:G} that
\begin{equation*}
\begin{split}
G_t G_t^\top &= \Delta_t U (\Delta_t U)^\top  + \Delta_t U \left(\int_{0}^ {t} U_s ds\right)^\top\bh + \bh\int_{0}^ {t} U_s ds\ (\Delta_t U)^\top\\
&\ +  \bh \int_{0}^ {t} U_s ds\left(\int_{0}^ {t} U_s ds\right)^\top \bh\\
&= \Delta_t U (\Delta_t U)^\top  + \int_{0}^ {t} \Delta_t U U_s^\top ds\ \bh + \bh\int_{0}^ {t} U_s (\Delta_t U)^\top ds\\
&\ + \bh \int_{0}^{t} \int_{0}^ {t} U_s U_u^\top du ds \bh
\end{split}
\end{equation*}
Taking expectation on both sides completes the proof.
Indeed, the first order term with respect to $\bh$ is
\begin{equation*}
\begin{split}
&\int_{0}^ {t} \left(\gamma(t-s) - \gamma(-s)\right)  ds\ \bh + \bh\int_{0}^ {t}\left( \gamma(s-t) - \gamma(s)\right) ds\\
 =& \int_0^t \left(\gamma(s) - \gamma(-s) \right) ds\ \bh + \bh \int_0^ t \left(\gamma(-s) - \gamma(s) \right)ds\\
 =& \int_0^t \left( \gamma(s) - \gamma(s)^\top \right)ds\ \bh + \bh \int_0^t \left( \gamma(s)^\top - \gamma(s)\right) ds\\
 =&  \Bh_t \bh + \bh \Bh_t^\top,
\end{split}
\end{equation*}
where $\gamma(t)$ is the cross-covariance matrix of $U$ given by \eqref{cross-covariance}. Computing other expectations similarly and rearranging terms gives us \eqref{riccati}.
\end{proof}

\subsection{Proofs of Theorem \ref{theo:pertur} and Theorem \ref{theo:fclt}}
\label{subsec:proof-3}
We begin with some preliminary notation. Let $\Bh$, $\Ch$, and $\Dh$ denote the coefficient matrices of the original CARE \eqref{riccati}, and let $\Phi = \Bh - \Ch \bh$. We define a linear operator $L: S^n \to S^n$ by
\begin{align*}
L(M) &= \Phi^\top M + M \Phi.
\end{align*}
The operator $L$ is bounded and invertible (see e.g. \cite{stewart1990} and \cite{sun1998perturbation}). In addition, we define linear operators $Q: S^n\to S^n$ and $P: \mathbb{R}^{n\times n} \to S^n$ by
\begin{align*}
Q(M) &= L^{-1} (\bh M\bh)\\
P(M) &= L^{-1}(\bh M+M^\top \bh).
\end{align*}
Since $L$ is bounded, also $Q$ and $P$ are bounded operators. We set
\begin{equation*}
l = \frac{1}{\Vert L^{-1}\Vert}, \quad p = \Vert P\Vert, \quad q = \Vert Q\Vert
\end{equation*}
and
\begin{align}
\label{epsilonetc}
\begin{split}
\epsilon_T &= \frac{1}{l} \Vert\Delta_T \Dh\Vert + p \Vert\Delta_T \Bh\Vert + q\Vert \Delta_T \Ch\Vert\\
\delta_T &= \Vert\Delta_T \Bh\Vert + \Vert\Delta_T \Ch\Vert \Vert\bh\Vert\\
g_T &= \Vert\Ch\Vert + \Vert\Delta_T \Ch\Vert\\
\epsilon_T^* &= \frac{2l\epsilon_T}{l-2\delta_T+\sqrt{(l-2\delta_T)^2 - 4lg_T\epsilon_T}}.
\end{split}
\end{align}
The following theorem is taken from \cite{sun1998perturbation}, but stated using our notation.

\begin{theorem}
\label{theorem:pertur}
Let $\bh$ be the unique positive semidefinite solution to the CARE \eqref{CARE}. Then if the coefficient matrices $\hat{\Dh}_T$ and $\hat{\Ch}_T$ of the perturbed CARE \eqref{perCARE} are positive semidefinite, and if
\begin{equation}
\label{condition}
\delta_T + \sqrt{lg_T\epsilon_T} < \frac{l}{2},
\end{equation}
then \eqref{perCARE} has a unique solution $\hat{\bh}_T\geq 0$ satisfying
$$\Vert\hat{\bh}_T - \bh\Vert \leq \epsilon_T^*.$$
\end{theorem}
\begin{rem}
Theorem \ref{theorem:pertur} holds for any unitarily invariant submultiplicative matrix norm $\Vert\cdot\Vert$, not just the spectral norm we are using. On the other hand, for our purposes the choice of the norm does not matter since in finite dimensions all norms are equivalent.
\end{rem}
Note that $\epsilon_T = \mathcal{O}(\Vert\Delta
_T \Dh\Vert + \Vert\Delta_T \Bh\Vert+ \Vert\Delta_T \Ch\Vert)$, where $\mathcal{O}$ denotes the usual Landau notation with meaning $X_T = \mathcal{O}(Y_T)$ if $|X_T| \leq c|Y_T|$ for some constant $c$. We also note that
Theorem \ref{theorem:pertur} gives us the following first-order perturbation bound
\begin{equation}
\label{firstorder}
\begin{split}
\Vert\hat{\bh}_T - \bh\Vert  \leq \epsilon_T + \mathcal{O}(\Vert\Delta_T \Dh\Vert^2 + \Vert\Delta_T \Bh\Vert^2 + \Vert\Delta_T \Ch\Vert^2).
\end{split}
\end{equation}
Recall that $\hat{\gamma}_T(s)$ and $\hat{\gamma}_{T,i,j}(s)$ denote some fixed estimators of $\gamma(s)$ and $\gamma_{i,j}(s)$ respectively. Next result gives bound \eqref{firstorder} in terms of cross-covariance estimators.
\begin{lemma}
\label{lemma:frobenius}
For $\epsilon_T$ given in \eqref{epsilonetc}, we have
\begin{equation*}
\begin{split}
\epsilon_T &\leq \left(2p+2qt\right) \int_0^t \Vert \hat{\gamma}_{T}(s)- \gamma(s)\Vert ds \\
&+ \frac{2}{l}\left(\Vert \hat{\gamma}_{T}(0)- \gamma(0)\Vert + \Vert \hat{\gamma}_{T}(t)- \gamma(t)\Vert\right).
\end{split}
\end{equation*}
In particular,
\begin{equation*}
\epsilon_T \leq \sup_{s\in[0,t]} \Vert\hat{\gamma}_{T}(s)- \gamma(s)\Vert \left(2pt+2qt^2 + \frac{4}{l}\right).
\end{equation*}
\begin{proof}
We have
\begin{equation*}
\begin{split}
\Delta_T \Bh &= \int_0 ^t \left( \hat{\gamma}_T(s) - \hat{\gamma}_T(s)^\top\right) ds - \int_0^t \left( \gamma(s) - \gamma(s)^\top \right) ds\\
&= \int_0^t \left( \hat{\gamma}_T(s) - \gamma(s) \right) ds + \int_0^t \left( \gamma(s)^\top - \hat{\gamma}_T(s)^\top \right)ds
\end{split}
\end{equation*}
implying
\begin{equation}
\label{||B||}
\Vert\Delta_T \Bh\Vert \leq 2 \int_0^t \Vert \hat{\gamma}_T(s) - \gamma(s) \Vert ds.
\end{equation}
Similarly, for
\begin{equation*}
\Delta_T \Ch = \int_0^t\int_0^t \hat{\gamma}_T(s-u) du ds - \int_0^t\int_0^t \gamma(s-u) du ds
\end{equation*}
we get
\begin{equation}
\label{||C||}
\begin{split}
\Vert\Delta_T \Ch\Vert &\leq \int_0^t\int_0^t \Vert\hat{\gamma}_T(s-u) - \gamma(s-u)\Vert du ds\\
&\leq 2t \int_0^t \Vert \hat{\gamma}_{T}(s) - \gamma(s) \Vert ds.
\end{split}
\end{equation}
Finally, for $\Delta_T \Dh$ we have
\begin{equation*}
\Delta_T \Dh = \hat{\Sigma}_{G_t} - \hat{\Sigma}_{U_t-U_0} - \Sigma_{G_t} + \Sigma_{U_t-U_0} = \Sigma_{U_t -U_0} - \hat{\Sigma}_{U_t-U_0},
\end{equation*}
where $\Sigma_X$ denotes the covariance matrix of a vector $X$.
Now
\begin{equation*}
\begin{split}
\Sigma_{U_t-U_0} = 2\gamma(0)- \gamma(t) - \gamma(t)^\top
\end{split}
\end{equation*}
giving us
\begin{equation}
\label{||D||}
\Vert\Delta_T \Dh\Vert \leq 2 \Vert \hat{\gamma}_{T}(0)- \gamma(0)\Vert+ 2 \Vert\hat{\gamma}_{T}(t)- \gamma(t)\Vert.
\end{equation}
Combining the estimates \eqref{||B||}, \eqref{||C||} and \eqref{||D||} concludes the proof.
\end{proof}
\end{lemma}

\begin{kor}
\label{kor:pertur}
Let $\bh$ be the unique positive definite solution to the CARE \eqref{CARE}. If the coefficient matrices $\hat{\Dh}_T$ and $\hat{\Ch}_T$ of the perturbed CARE \eqref{perCARE} are positive semidefinite, and if
\begin{equation}
\label{condition2}
\delta_T + \sqrt{lg_T\epsilon_T} < \frac{l}{2},
\end{equation}
then \eqref{perCARE} has a unique positive semidefinite solution $\hat{\bh}_T$. Moreover,
\begin{equation*}
\Vert\hat{\bh}_T - \bh\Vert \leq  c \sup_{s\in[0,t]} \Vert\hat{\gamma}_{T}(s)- \gamma(s)\Vert + \mathcal{O}\left(\sup_{s\in[0,t]} \Vert\hat{\gamma}_{T}(s)- \gamma(s)\Vert^2\right),
\end{equation*}
where the constant $c$ depends on $n$, $\bh$, $t$ and the cross-covariance of $G$.

\begin{proof}
The claim follows directly from the bound \eqref{firstorder}, Lemma \ref{lemma:frobenius} and the proof of  Lemma \ref{lemma:frobenius}.
\end{proof}
\end{kor}
We are now in position to proof our consistency result, Theorem \ref{theo:pertur}.
\begin{proof}[Proof of Theorem \ref{theo:pertur}]
We first pick $\rho>0$ small enough such that if
\begin{equation}
\label{assumption}
\sup_{s\in[0,t]} \Vert\hat{\gamma}_{T}(s)- \gamma(s)\Vert  \leq \rho,
\end{equation}
then $\delta_T + \sqrt{lg_T\epsilon_T} < \frac{l}{2}$. Since $\Ch,\Dh>0$, we have that $\hat{\Ch}_T,\hat{\Dh}_T$ are positive definite whenever $\rho$ is chosen small enough.
 Moreover, in this case
Corollary \ref{kor:pertur} gives 
\begin{equation}
\label{eq:H-bound}
\Vert\hat{\bh}_T - \bh\Vert \leq  \tilde{C} \sup_{s\in[0,t]} \Vert\hat{\gamma}_{T}(s)- \gamma(s)\Vert.
\end{equation}
Next, let $\varepsilon>0$ be arbitrary such that $\frac{\varepsilon}{\tilde{C}} \leq \rho$ and set
\begin{equation*}
A_{T,\varepsilon} \coloneqq \left\{\omega : \sup_{s\in[0,t]} \Vert\hat{\gamma}_{T}(s)- \gamma(s)\Vert \leq \frac{\varepsilon}{\tilde{C}} \right\}.
\end{equation*}
Now a unique positive semidefinite solution $\hat{\bh}_T$ to \eqref{perCARE} exists for  $\omega \in A_{T,\varepsilon}$, and we have \eqref{eq:H-bound}. Moreover, assumption
$\sup_{s\in[0,t]} \Vert\hat{\gamma}_{T}(s)- \gamma(s)\Vert \overset{\p}{\longrightarrow} 0$ implies that for any $\xi>0$ there exists $T_{\varepsilon,\xi}$ such that for every $T \geq T_{\varepsilon, \xi}$
we have $\p(A_{T,\varepsilon}) \geq 1- \frac{\xi}{2}$. Thus we can conclude
\begin{equation*}
\begin{split}
\p\left(\Vert\hat{\bh}_T - \bh\Vert > \varepsilon\right) &\leq \p\left(\mathbbm{1}_{A_{T,\varepsilon}} \Vert \hat{\bh}_T - \bh\Vert > \varepsilon\right) + \p\left(\mathbbm{1}_{A_{T,\varepsilon}^c}\Vert  \hat{\bh}_T - \bh\Vert > \varepsilon\right) \\
&\leq \p\left( \sup_{s\in[0,t]} \Vert\hat{\gamma}_{T}(s)- \gamma(s)\Vert > \frac{\varepsilon}{\tilde{C}}\right) + \p(A_{T,\varepsilon}^c)  \\
& \leq \frac{\xi}{2} + \frac{\xi}{2} = \xi.
\end{split}
\end{equation*}
This concludes the proof.
\end{proof}
Before proving Theorem \ref{theo:fclt} we recall an auxiliary lemma, taken from \cite{horn1991}.

\begin{lemma}
\label{lemma:krosum}
Let $E$ and $F$ be square matrices of sizes $m$ and $n$, respectively. Then all the eigenvalues of the Kronecker sum
\begin{equation*}
E \oplus F = (I_n \otimes E) + (F \otimes I_m)
\end{equation*}
are of the form $\lambda_i + \lambda_j$, where $\lambda_i$ is an eigenvalue of $E$ and $\lambda_j$ is an eigenvalue of $F$.
\end{lemma}

\begin{proof}[Proof of Theorem \ref{theo:fclt}]
\textbf{Item (1)}:\\
Recall that
\begin{align*}
\Delta_T \Dh &= 2(\hat{\gamma}_T(0) - \gamma(0)) + \gamma(t) - \hat{\gamma}_T(t) + \gamma(t)^\top - \hat{\gamma}_T(t)^\top\\
\Delta_T  \Bh &= \int_0^t \hat{\gamma}_T(s) - \gamma(s) ds + \int_0^t \gamma(s)^\top - \hat{\gamma}_T(s)^\top ds.
\end{align*}
Similarly, using
\begin{equation*}
\begin{split}
\int_0^ t \int_0^ t \gamma(u-s) du ds = \int_0^ t \int_0^s \gamma(u-s) du ds + \int_0^t \int_s^t \gamma(u-s) du ds,
\end{split}
\end{equation*}
where
\begin{equation*}
\begin{split}
 \int_0^ t \int_0^s \gamma(u-s) du ds &= \int_0^ t \int_0^s \gamma(s-u)^\top du ds = \int_0^t \int_0^s \gamma(z) ^\top dz ds\\
&= \int_0^t \int_z^t  \gamma(z)^\top ds dz = \int_0^ t (t-z) \gamma(z)^\top dz
\end{split}
\end{equation*}
and
\begin{equation*}
\int_0^t \int_s^t \gamma(u-s) du ds = \int_0^ t (t-z) \gamma(z) dz,
\end{equation*}
we get
\begin{equation*}
\begin{split}
\Delta_T \Ch &= \int_0^t \int_0^ t \hat{\gamma}_T(s-u) - \gamma(s-u) du ds \\
&= \int_0^ t (t-s) (\hat{\gamma}_T(s) - \gamma(s)) ds + \int_0^ t (t-s) (\hat{\gamma}_T(s)^\top - \gamma(s)^\top) ds.
\end{split}
\end{equation*}
Now by assumption, we have
\begin{equation*}
l(T)\vect (\hat{\gamma}_T(s)- \gamma(s)) \overset{\text{law}}{\longrightarrow} X_s,
\end{equation*}
where $X = (X_s)_{s\in[0,t]}$ is an $n^2$-dimensional stochastic process with continuous paths. Now it is clear that
the linear mapping $L_1: C[0,t]^{n^2} \to \mathbb{R}^{3n^2}$ defined by
\begin{equation*}
L_1(X) = \begin{bmatrix*}
\int_0^ t (t-s)(X_s + \tilde{X}_s) ds\\
\int_0^t X_s -\tilde{X}_s ds\\
2X_0 - X_t -\tilde{X}_t,
\end{bmatrix*}
\end{equation*}
where $\tilde{X}_s$ denotes the permutation of the elements of $X_s$ that corresponds to the order of the elements of $\vect(\gamma(s)^\top)$, is a continuous operator.
Thus we may apply continuous mapping theorem
to conclude that
\begin{equation*}
L_1(l(T)\vect(\hat{\gamma}_T(s)-\gamma(s))) = l(T)\vect(\Delta_T \Ch, \Delta_T \Bh, \Delta_T \Dh) \overset{\text{law}}{\longrightarrow} L_1(X).
\end{equation*}
\textbf{Item (2):}\\
As in the proof of Theorem \ref{theo:pertur}, set
\begin{equation*}
A_T \coloneqq \left\{\omega : \sup_{s\in[0,t]} \Vert\hat{\gamma}_{T}(s)- \gamma(s)\Vert \leq \rho\right\},
\end{equation*}
where $\rho$ is chosen as in the proof of Theorem \ref{theo:pertur}. Then the unique (positive semidefinite) solution $\hat{\bh}_T$ to the perturbed CARE \eqref{perCARE} exists for all $\omega \in A_T$. Let $\Delta_T \bh = \hat{\bh}_T - \bh$. We write 
\begin{equation}
\label{twoasymptotics}
l(T)\vect(\Delta_T \bh) = l(T)\ind \vect(\Delta_T \bh)+ l(T)\indi \vect(\Delta_T \bh).
\end{equation}
Since \eqref{functionalcovariance} implies \eqref{covarianceinP}, Theorem \ref{theo:pertur} implies $\vect(\Delta_T \bh) \overset{\p}{\longrightarrow} 0$. Moreover, we have
\begin{equation*}
\p(l(T)\indi > \varepsilon) \leq \p(A_T^c) \to 0
\end{equation*}
for every $\varepsilon >0$. Thus the second term in \eqref{twoasymptotics} converges to zero in probability, and hence, by Slutsky's theorem, it suffices to consider the first term in \eqref{twoasymptotics}. For this, we first observe that, by the proof of Theorem \ref{theorem:pertur} in \cite{sun1998perturbation}, we have
\begin{equation*}
(B - C\bh)^\top \Delta_T \bh + \Delta_T \bh (\Bh-\Ch \bh) = -E + h_1(\Delta_T \bh) + h_2 (\Delta_T \bh),
\end{equation*}
where
\begin{align*}
E_T &= \Delta_T \Dh +  \Delta_T \Bh^\top \bh + \bh \Delta_T \Bh - \bh \Delta_T \Ch \bh\\
h_1(\Delta_T \bh) &= -[(\Delta_T \Bh - \Delta_T \Ch \bh)^\top \Delta_T \bh + \Delta_T \bh (\Delta_T \Bh - \Delta_T \Ch \bh)]\\
h_2(\Delta_T \bh) &= \Delta_T \bh (\Ch+\Delta_T \Ch) \Delta_T \bh.
\end{align*}
Recall the notation $\Phi = \Bh - \Ch \bh$. Now, by compatibility of vectorization operator and Kronecker product we obtain
\begin{small}
\begin{equation*}
\begin{split}
 \vect \left((\Bh - \Ch \bh)^\top \Delta_T \bh + \Delta_T \bh (\Bh-\Ch \bh) \right) &= \vect (\Phi^\top \Delta_T \bh + \Delta_T \bh \Phi)\\
&= (I \otimes \Phi^\top) \vect (\Delta_T \bh) + (\Phi^\top \otimes  I) \vect (\Delta_T \bh)\\
&= ( I \otimes \Phi^\top + \Phi^\top \otimes I) \vect(\Delta_T \bh)\\
&= (\Phi^\top \oplus \Phi^\top) \vect(\Delta_T \bh),
\end{split}
\end{equation*}
\end{small}
where $\oplus$ denotes the Kronecker sum. By Lemma \ref{lemma:uniqueness} the matrix $\Phi$ is stable, i.e. all its eigenvalues are in the open left half-plane. Hence, by Lemma \ref{lemma:krosum}, $(\Phi^\top \oplus \Phi^\top)$ is invertible, and consequently
\begin{equation*}
\begin{split}
&l(T)\ind \vect(\Delta_T \bh)\\
 =&\ l(T) \ind  (\Phi^\top \oplus \Phi^\top)^{-1}\left(-\vect(E_T) + \vect(h_1(\Delta_T \bh)) + \vect(h_2(\Delta_T \bh)) \right).
\end{split}
\end{equation*}
By compatibility of vectorization and Kronecker product, the terms of\\ $l(T)\ind \vect (h_1(\Delta_T \bh))$ are given by
\begin{align*}
l(T) \ind \vect (\Delta_T \bh \Delta_T \Bh) &= l(T) \ind ( I \otimes \Delta_T \bh) \vect (\Delta_T \Bh)\\
l(T) \ind \vect (\Delta_T \Bh^\top \Delta_T \bh) &= l(T) \ind (\Delta_T \bh \otimes I) \vect (\Delta_T \Bh^\top) \\
l(T) \ind \vect (\bh \Delta_T \Ch \Delta_T \bh) &= l(T) \ind ( \Delta_T\bh  \otimes \bh) \vect (\Delta_T \Ch)\\
l(T) \ind \vect (\Delta_T \bh \Delta_T \Ch \bh) &= l(T) \ind ( \bh \otimes \Delta_T \bh) \vect ( \Delta_T \Ch).\\
\end{align*}
By item (1), Theorem \ref{theo:pertur}, and Slutsky's theorem all these terms converge to zero in probability. Similarly, for the asymptotically dominant term of \\$l(T)\ind \vect(h_2(\Delta_T \bh))$, we get
\begin{equation*}
\ind \Vert \Delta_T \bh \Ch \Delta_T \bh\Vert \leq \ind \Vert \Delta_T \bh\Vert^ 2 \Vert \Ch\Vert.
\end{equation*}
Now on the set $A_T$ we have, by \eqref{firstorder} and item (1), that
\begin{equation*}
\ind\Vert\Delta_T \bh\Vert = \ind \mathcal{O}(\Vert\Delta_T \Dh\Vert + \Vert\Delta_T \Bh\Vert+ \Vert\Delta_T \Ch\Vert).
\end{equation*}
Thus $l(T)\ind \vect(h_2(\Delta_T \bh))$ converges to zero in probability as well, and it suffices to study asymptotics of
\begin{equation}
\label{eq:key-term-CLT}
l(T) \ind  (\Phi^\top \oplus \Phi^\top)^{-1}\left(-\vect(E_T)\right).
\end{equation}
For this we write
\begin{multline*}
-l(T)\ind \vect(E_T) =  l(T) \ind \vect \left( \bh \Delta_T \Ch \bh - \Delta_T \Dh - \Delta_T \Bh^\top \bh - \bh \Delta_T \Bh\right)
\\ = l(T) \ind \Big((\bh \otimes \bh) \vect (\Delta_T \Ch) - \vect(\Delta_T \Dh) - (\bh \otimes I) \vect (\Delta_T \Bh^\top)\\
 - (I \otimes \bh) \vect (\Delta_T \Bh)\Big) .
\end{multline*}
We define a linear function $f:\mathbb{R}^ {3n^2} \rightarrow \mathbb{R}^{n^2}$ by
\begin{small}
\begin{equation}
\label{functionf}
f (\vect (\Ch, \Bh, \Dh)) =   (\bh \otimes \bh) \vect (\Ch) - \vect(\Dh) - (\bh \otimes I) \vect (\Bh^\top) - (I \otimes \bh) \vect (\Bh).
\end{equation}
\end{small}
Then the usual Delta method and Slutsky's theorem implies
\begin{small}
\begin{equation*}
\begin{split}
-l(T)\ind \vect(E_T) &= l(T) \ind \left( f(\vect(\hat{\Ch}_T, \hat{\Bh}_T, \hat{\Dh}_T)) - f (\vect (\Ch, \Bh, \Dh)\right) \overset{\text{law}}{\longrightarrow} L_2^\star L_1(X),
\end{split}
\end{equation*}
\end{small}
where the linear mapping $L_2^\star$ is given by the Jacobian (i.e. matrix representation) of the function $f$ defined in \eqref{functionf}. It remains to apply continuous mapping theorem to \eqref{eq:key-term-CLT} to conclude that
\begin{equation*}
l(T)\ind \vect(\Delta_T \bh) \overset{\text{law}}{\longrightarrow} (\Phi^\top \oplus \Phi^\top)^{-1}L_2^\star L_1(X) \eqqcolon L_2L_1(X).
\end{equation*}
This completes the proof.
\end{proof}

%%%%%%%%%%%%%%%%%%%%
\subsection{Representation of cross-covariance matrix $\gamma(t)$}
\label{subsec:proof-4}
Our main results relies on the cross-covariance matrix $\gamma(s)$ of $U$. However, in many models one assumes that only the variance of the noise $G$ is known. In this subsection we give several representations for $\gamma(s)$ in terms of the variance matrix $v(t) = \e(G_t G_t^\top)$, provided that $G$ has independent components. In particular, then $\e(G_tG_s^\top)$ is a diagonal matrix for all $t,s$, and satisfies
\begin{equation*}
\e(G_tG_s^\top) = \frac{1}{2}(v(t)+v(s)-v(t-s)).
\end{equation*}
Our first representation is the following.
\begin{lemma}
\label{lma:crosscov-rep-1}
Let $G$ have independent components. Then
\begin{multline}
\label{covofU}
\gamma(r) = \frac{e^{-\bh r} \bh}{2}\left( \int_{-\infty}^r e^{\bh x} v(x) dx - \int_{-\infty}^r \int_{-\infty}^0 e^{\bh x} e^{\bh s} v(x) e^{\bh s} \bh ds dx \right.\\  \left. - \int_{r}^\infty \int_{-\infty}^{r-x} e^{\bh x} e^{\bh s} v(x) e^{\bh s}\bh ds dx \right)
+  \frac{1}{2} \int_r^\infty v(x) e^{\bh (r-x)} \bh dx - \frac{1}{2} v(r).
\end{multline}
\begin{proof}
Using representation
$
U_t = G_t -e^{-\bh t}\bh \int_{-\infty}^t e^{\bh s} G_s ds
$
gives us
\begin{equation*}
U_rU_0^\top = -\int_{-\infty}^0 G_r G_s^\top e^{\bh s}\bh ds + e^{-\bh r} \bh \int_{-\infty}^ 0 \int_{-\infty}^r e^{\bh u} G_u G_s^\top e^ {\bh s} \bh du ds. 
\end{equation*}
Taking expectation and using Fubini's theorem thus yields
\begin{equation*}
\begin{split}
\gamma(r) &= -\frac{1}{2} \int_{-\infty}^ 0 (v(r) + v(s) - v(r-s)) e^{\bh s} \bh ds\\
 &+ e^ {-\bh r} \frac{\bh}{2} \int_{-\infty}^0 \int_{-\infty}^ r  e^{\bh u}(v(u) + v(s) - v(u-s)) e^{\bh s} \bh du ds.
\end{split}
\end{equation*}
Here
\begin{small}
\begin{equation*}
e^{-\bh r}\frac{\bh}{2} \int_{-\infty}^ 0 \int_{-\infty}^ r e^{\bh u} v(s) e^ {\bh s} \bh du ds = e^ {-\bh r} \frac{1}{2} \int_{-\infty}^ 0 e^{\bh r} v(s) e^{\bh s} \bh ds = \frac{1}{2} \int_{-\infty}^0 v(s) e^ {\bh s} \bh ds
\end{equation*}
\end{small}
leading to
\begin{small}
\begin{equation*}
\gamma(r)  =  -\frac{1}{2} \int_{-\infty}^ 0 (v(r)  - v(r-s)) e^{\bh s} \bh ds + e^ {-\bh r} \frac{\bh}{2} \int_{-\infty}^0 \int_{-\infty}^ r  e^{\bh u}(v(u)  - v(u-s)) e^{\bh s} \bh du ds.
\end{equation*}
\end{small}
By the change of variable $x = u-s,$ we get

\begin{equation*}
\begin{split}
 \int_{-\infty}^0 \int_{-\infty}^ r  e^{\bh u}v(u-s) e^{\bh s} \bh du ds &=  \int_{-\infty}^ 0 \int_{-\infty}^ {r-s} e^{\bh(x+s)} v(x) e^{\bh s} \bh dx ds\\
&=  \int_{-\infty} ^\infty \int_{-\infty}^{\text{min}\{0, r-x\}} e^{\bh x} e^{\bh s} v(x) e^{\bh s} \bh ds dx\\
&=  \int_{-\infty}^ r \int_{-\infty}^ 0 e^{\bh x} e^ {\bh s} v(x) e^{\bh s} \bh ds dx \\
&+ \int_{r}^\infty \int_{-\infty}^ {r-x} e^{\bh x} e^ {\bh s} v(x) e^{\bh s} \bh ds dx. 
\end{split}
\end{equation*}
Finally, we have the identities 

\begin{align*}
\int_{-\infty}^ 0 v(r-s) e^{\bh s} \bh ds &= \int_r^ \infty v(x) e^{\bh(r-x)} \bh dx\\
\int_{-\infty}^ 0 \int_{-\infty}^ r e^{\bh u} v(u) e^{\bh s} \bh du ds &= \int_{-\infty}^r e^{\bh x} v(x) dx\\
\int_{-\infty}^ 0 v(r) e^{\bh s} \bh ds &= v(r).
\end{align*}
Combining all the results above gives us \eqref{covofU}.
\end{proof}
\end{lemma}

\begin{lemma}
\label{lemma:covofU2}
Let $G$ have independent components. Then
\begin{small}
\begin{equation}
\label{nonidentical}
\gamma(r) = \frac{\bh}{2} \int_{-\infty}^ 0 \int_{-\infty}^ 0 e^{\bh x} \left( v(x+r) - v(r) + v(r-s) - v(x+r-s)\right) e^{\bh s} \bh dx ds.
\end{equation}
\end{small}
\begin{proof}
The expression 
\begin{small}
\begin{equation*}
\gamma(r)  =  -\frac{1}{2} \int_{-\infty}^ 0 (v(r)  - v(r-s)) e^{\bh s} \bh ds + e^ {-\bh r} \frac{\bh}{2} \int_{-\infty}^0 \int_{-\infty}^ r  e^{\bh u}(v(u)  - v(u-s)) e^{\bh s} \bh du ds
\end{equation*}
\end{small}
can be written as
\begin{equation*}
\gamma(r) = -\frac{1}{2} \int_{-\infty}^0 e^{-\bh r} \bh \int_{-\infty}^ r e^{\bh u} ( v(r) - v(u) - v(r-s) + v(u-s)) e^{\bh s} \bh du  ds,
\end{equation*}
from which the claim follows by the change of variable $u-r = x$.
\end{proof}
\end{lemma}
\begin{rem}
If the matrices $v(t)$ and $\bh$ commute for every $t\in \mathbb{R}$, we obtain even simpler expression
\begin{equation*}
\gamma(r) = \frac{\bh}{4}\int_{-\infty}^0 e^ {\bh x} \left( v(x+r)  + v(r-x) -2v(r)\right) dx.
\end{equation*}
In particular, this is the case if $G$ consists of independent processes with equal variances. 
\end{rem}

\subsection{Proofs related to Gaussian examples}
\label{subsec:proof-5}
This section is devoted to the proofs of Proposition \ref{prop:Gaussian-consistency}, Theorem \ref{theorem:Gaussian-clt}, and Corollary \ref{kor:fbm-clt}. We begin with the proof of Proposition \ref{prop:Gaussian-consistency}. 
\begin{proof}[Proof of Proposition \ref{prop:Gaussian-consistency}]
We have
$$
\hat{\gamma}_{T,i,j}(s)-\gamma_{i,j}(s) = \frac{1}{T} \int_0^T \left( U^{(i)}_{r+s} U^{(j)}_r - \e \left(U^{(i)}_{s} U^{(j)}_0\right)\right)dr
$$
implying, with straightforward computations, that
$$
\e \left[\hat{\gamma}_{T,i,j}(s)-\gamma_{i,j}(s)\right]^2 \leq \frac{2}{T} \int_0^T |\gamma_{i,j}(r+s)|^2dr.
$$
Now $\Vert \gamma(r)\Vert \to 0$ implies $|\gamma_{i,j}(r)| \to 0$ as well, and thus, for each fixed $s$, 
$$
|\hat{\gamma}_{T,i,j}(s)-\gamma_{i,j}(s)| \to 0
$$
in $L^2$. These further implies 
$$
\sup_{s\in[0,t]}\e\Vert\hat{\gamma}_T(s)-\gamma(s)\Vert^2 \to 0,
$$ from which we conclude that conditions of Remark \ref{rem:sufficient} are satisfied. The claim then follows.
\end{proof}
In order to prove Theorem \ref{theorem:Gaussian-clt} one needs to prove the convergence of finite dimensional distributions and tightness. For the latter we present the following result that might be interesting on its own. In the sequel, we use the short notation
\begin{equation*}
F_T(\tau) = l(T) \vect  (\hat{\gamma}_T(\tau)- \gamma(\tau)).
\end{equation*}
\begin{prop}
\label{prop:tightness}
Suppose that $\gamma(r)$ is differentiable for almost all $r$ and 
$$
\max\left(\Vert\gamma'(r)\Vert,\Vert \gamma(r)\Vert\right) \leq h(r)
$$
for some non-increasing function $h(r)$ such that, for some $K>0$, we have $h(r) \in L^1([0,K])$ and
$$
\int_K^T h(r)^2 dr = \mathcal{O}\left(\frac{T}{l(T)^2}\right), \quad T>K.
$$
Then there exists $T_0$ such that for all $\tau,s\in[0,t]$, all $a\in \mathbb{R}^{n^2}$ and all $p\geq 2$, we have
\begin{equation*}
\e \left| a^{\top}(F_T(\tau) - F_T(s)) \right|^p \leq c|\tau -s|^{\frac{p}{2}},\quad T\geq T_0,
\end{equation*}
where $c$ depends only on $p$, $t$, $\gamma$, and $a^{\top}$.
\end{prop}
\begin{proof}
We have
\begin{equation*}
F_T(\tau) - F_T(s) = \frac{l(T)}{T} \int_0^T \vect (U_{u+\tau} U_u^\top - \gamma(\tau) - U_{u+s} U_u^\top + \gamma(s)) du.
\end{equation*}
First we note that it suffices to prove the claim only for $p=2$. Indeed, since $U$ is Gaussian, the expression $a^\top (F_T(\tau) - F_T(s))$ belongs to the so-called second Wiener chaos (for details, see e.g. \cite{Janson-1997}), implying the hypercontractivity property 
$$
\e|a^\top (F_T(\tau) - F_T(s))|^p \leq c_p \left[\e|a^\top (F_T(\tau) - F_T(s))|^2 \right]^{\frac{p}{2}}.
$$ 
Thus, let $p=2$. We have
$$
|a^\top (F_T(\tau) - F_T(s))|^2 \leq c_a \Vert F_T(\tau) - F_T(s)\Vert^2,
$$
where 
\begin{equation*}
\Vert F_T(\tau) - F_T(s)\Vert^2 = \frac{l(T)^2}{T^2} \sum_{i,j = 1}^n \left( \int_0^T \left( U_{u+\tau} U_u^\top - \gamma(\tau) - U_{u+s} U_u^\top + \gamma(s)\right)_{i,j} du \right)^2.
\end{equation*}
Here
\begin{equation*}
\begin{split}
&\left( \int_0^T (U_{u+\tau} U_u^\top - \gamma(\tau)-U_{u+s}U_u^ \top + \gamma(s))_{i,j} du \right)^ 2\\
&= \int_0^T U_{u+\tau}^{(i)} U_u^{(j)} - \gamma_{i,j}(\tau) - U_{u+s}^{(i)} U_u^{(j)} + \gamma_{i,j}(s) du\\ &\times\int_0^T U_{v+\tau}^{(i)} U_v^{(j)} - \gamma_{i,j}(\tau) - U_{v+s}^{(i)} U_v^{(j)} + \gamma_{i,j}(s) dv.
\end{split}
\end{equation*}
Taking expectation and with some straightforward computations, we get
\begin{equation*}
\begin{split}
&\e\left( \int_0^T (U_{u+\tau} U_u^\top - \gamma(\tau)-U_{u+s}U_u^ \top + \gamma(s))_{i,j} du \right)^ 2\\ 
&= \int_0^ T (T-x) \gamma_{j,j,}(x) ( \gamma_{i,i}(x) - \gamma_{i,i} (x+\tau-s)) dx \\
& + \int_0^ T (T-x) \gamma_{j,j}(x)( \gamma_{i,i}(x) - \gamma_{i,i}(-x+\tau-s)) dx\\
& + \int_0^ T (T-x) \gamma_{j,j,}(x) ( \gamma_{i,i}(x) - \gamma_{i,i} (x+s-\tau)) dx \\
& + \int_0^ T (T-x) \gamma_{j,j}(x)( \gamma_{i,i}(x) - \gamma_{i,i}(-x+s-\tau)) dx\\
& + \int_0^ T (T-x) \gamma_{i,j}(x+\tau) ( \gamma_{i,j}(-x+\tau) - \gamma_{i,j} (-x+s)) dx \\
& + \int_0^ T (T-x) \gamma_{i,j}(-x+\tau)( \gamma_{i,j}(x+\tau) - \gamma_{i,j}(x+s)) dx\\
& + \int_0^ T (T-x) \gamma_{i,j}(x+s) ( \gamma_{i,j}(-x+s) - \gamma_{i,j} (-x+\tau)) dx \\
& + \int_0^ T (T-x) \gamma_{i,j}(-x+s)( \gamma_{i,j}(x+s) - \gamma_{i,j}(x+\tau)) dx.
\end{split}
\end{equation*}
Thus it suffices to show that all eight terms, when multiplied with $\frac{l(T)^ 2}{T^ 2}$, admit a bound of the form $C|\tau - s|$. We show how the first term can be treated, while the rest can be shown with similar arguments. Without loss of generality, let $s<\tau$ and denote $u=\tau-s$. For the first term above, we apply the mean value theorem to obtain 
\begin{equation*}
\begin{split}
&\left|\int_0^ T (T-x) \gamma_{j,j,}(x) ( \gamma_{i,i}(x) - \gamma_{i,i} (x+u)) dx\right|  \\
& \leq T u \left(\int_0^ K |\gamma_{j,j}(x)| \sup_{x < \xi < x+u} |\gamma'_{i,i} (\xi)| dx  + \int_K^ T |\gamma_{j,j}(x)| \sup_{x < \xi < x+u} |\gamma'_{i,i} (\xi)| dx \right).
\end{split}
\end{equation*}
By assumption, 
$$
\sup_{x < \xi < x+u} |\gamma'_{i,i} (\xi)| \leq \sup_{x < \xi < x+u}\Vert \gamma'(\xi)\Vert \leq h(x).
$$
Since $|\gamma_{j,j}(x)|$ is bounded and $h(x) \in L^1([0,K])$, we get
$$
\frac{l(T)^2}{T^2}\cdot T u \int_0^ K |\gamma_{j,j}(x)| \sup_{x < \xi < x+u} |\gamma'_{i,i} (\xi)| dx \leq cu \frac{l(T)^2}{T}.
$$
Now, the best possible rate $l(T)$ that one can have is $\sqrt{T}$, giving $\sup_{T\geq T_0}\frac{l(T)^2}{T} < \infty$. Similarly, we have
$$
\int_K^ T |\gamma_{j,j}(x)| \sup_{x < \xi < x+u} |\gamma'_{i,i} (\xi)| dx \leq c\int_K^T h(x)^2dx \leq c \frac{T}{l(T)^2}
$$
by assumption. Treating the rest of the terms similarly concludes the proof.
\end{proof}
The proof of Theorem \ref{theorem:Gaussian-clt} is now rather straightforward.
\begin{proof}[Proof of Theorem \ref{theorem:Gaussian-clt}]
By Cramer-Wold device it suffices to prove the convergence of linear combinations, and then the tightness follows from Proposition \ref{prop:tightness}. In order to obtain convergence of multidimensional distributions, we have to prove that all the combinations of the form
$$
\sqrt{T} \sum_{k=1}^d a_k^\top \vect  (\hat{\gamma}_T(\tau_k)- \gamma(\tau_k))
$$
converges towards $\sum_{k=1}^d a_k^\top X_{\tau_k}$, where $a_k$ are some $n^2$-dimensional vectors. Now the above expression can be written as 
$$
\frac{1}{\sqrt{T}} \int_0^T \sum_{k=1}^d \sum_{i,j=1}^n a_k^{(i,j)} \left[U^{(i)}_{r+\tau_k}U^{(j)}_r-\gamma_{i,j}(\tau_k)\right]dr,
$$
and thus the convergence towards a Gaussian limit follows from the continuous time Breuer-Major Theorem (see, e.g. \cite{cont-BM} and references therein) together with the fact that now $\int_0^\infty \Vert \gamma(s)\Vert^2 ds < \infty$. This completes the proof.
\end{proof}
Finally, we verify the result for multidimensional fractional Ornstein-Uhlenbeck process.
\begin{proof}[Proof of Corollary \ref{kor:fbm-clt}]
Let $H_{max} = \max_{1\leq i\leq n}H_i$ and $H_{min} = \min_{1\leq i\leq n}H_i$. We prove that asymptotically, as $t\to \infty$, we have 
\begin{equation}
\label{eq:fbm-gamma}
\Vert \gamma(t)\Vert = \mathcal{O}\left(t^{2H_{max}-2}\right),
\end{equation}
\begin{equation}
\label{eq:fbm-gamma-derivative}
\Vert \gamma'(t)\Vert = \mathcal{O}\left(t^{2H_{max}-2}\right),
\end{equation}
and that $\Vert \gamma'(t)\Vert =\mathcal{O}\left(\max\left(t^{2H_{min}-1},1\right)\right)$ for $t\leq K$ with $K$ fixed. 
Since $H_{max} < \frac34$ and $H_{min}>0$, the statement then follows from Theorem \ref{theorem:Gaussian-clt}. For this, let $g$ be an auxiliary function such that $\frac{g(r)}{r} \to 0$ and
$\frac{\ln r}{g(r)}\to 0$. Then it follows that $r^{4-2H_{max}}e^{-cg(r)} \to 0$ for all $c>0$. We begin by showing \eqref{eq:fbm-gamma}. 
We divide the expression \eqref{nonidentical} for $\gamma(r)$ into 
\begin{small}
\begin{align}
&\frac{\bh}{2} \int_{-\infty}^ {-g(r)} \int_{-\infty}^ {-g(r)} e^{\bh x} \left( v(x+r) - v(r) + v(r-s) - v(x+r-s)\right) e^{\bh s} \bh dx ds  \label{partone}\\
&\frac{\bh}{2} \int_{-g(r)}^ {0} \int_{-g(r)}^ {0} e^{\bh x} \left( v(x+r) - v(r) + v(r-s) - v(x+r-s)\right) e^{\bh s} \bh dx ds. \label{parttwo}
\end{align}
\end{small}
Since $r$ is large, we obtain for \eqref{parttwo} that
\begin{equation*}
\begin{split}
&v_{i,i}(x+r)  - v_{i,i}(r) + v_{i,i}(r-s) - v_{i,i}(x+r-s)\\
 =& |x+r|^ {2\h_i} - r^{2\h_i} + |r-s|^ {2\h_i} - |x+r-s|^ {2\h_i}\\
 =& r^ {2\h_i}((1+\frac{x}{r})^{2\h_i} + (1-\frac{s}{r})^ {2\h_i} -1 - (1+ \frac{x-s}{r})^ {2\h_i})\\
 =& r^ {2\h_i}\left( \mathcal{O} \left(\left(\frac{x}{r}\right)^ 2\right) + \mathcal{O}\left(\left(\frac{s}{r}\right)^ 2\right)\right), \quad\text{when } \h_i \neq \frac{1}{2}
\end{split}
\end{equation*}
and
\begin{equation*}
v_{i,i}(x+r)  - v_{i,i}(r) + v_{i,i}(r-s) - v_{i,i}(x+r-s) = 0, \quad\text{when } \h_i = \frac{1}{2}. 
\end{equation*}
Hence
\begin{multline*}
\frac{\bh}{2} \int_{-g(r)}^ {0} \int_{-g(r)}^ {0} e^{\bh x} \left( v(x+r) - v(r) + v(r-s) - v(x+r-s)\right) e^{\bh s} \bh dx ds\\
=  \frac{\bh}{2} r^{2\h_{max}}\int_{-g(r)}^ {0} \int_{-g(r)}^ {0} e^{\bh x} \left( \mathcal{O} \left(\left(\frac{x}{r}\right)^ 2\right) + \mathcal{O}\left(\left(\frac{s}{r}\right)^ 2\right)\right)\\
 \mathrm{diag}\left(\mathbbm{1}_{\h_i \neq \frac{1}{2}} r^{2(\h_i-\h_{max})}\right) e^{\bh s} \bh dx ds,
\end{multline*}
and taking the norm gives
\begin{small}
\begin{equation*}
\begin{split}
&\left|\left|\frac{\bh}{2} \int_{-g(r)}^ {0} \int_{-g(r)}^ {0} e^{\bh x} \left( v(x+r) - v(r) + v(r-s) - v(x+r-s)\right) e^{\bh s} \bh dx ds\right|\right|\\
&\leq Cr^ {2 \h_{max} -2} \int_{-g(r)}^ {0} \int_{-g(r)}^ {0} e^{\lambda_{min}(x+s)} (x^2+s^ 2) dxds = \mathcal{O}(r^ {2\h_{max}-2}).
\end{split}
\end{equation*}
\end{small}
It remains to study \eqref{partone}. We have four terms for which one satisfies 
\begin{small}
\begin{equation}
\label{insig3}
\begin{split}
\left|\left| \int_{-\infty}^ {-g(r)} \int_{-\infty}^ {-g(r)} e^{\bh x} v(r-s) e^{\bh s} \bh dxds \right|\right|  & \leq C e^{-\lambda_{min} g(r)} \int_{-\infty}^ {-g(r)} e^{\lambda_{min} s} (r-s)^ 2 ds\\
&= \mathcal{O}\left(e^{-\lambda_{min} g(r)}r^ 2\right)\\
&= \mathcal{O}\left(r^{2\h_{max}-2}\right)
\end{split}
\end{equation}
\end{small} by our choice of $g(r)$. Treating the remaining three terms similarly proves \eqref{eq:fbm-gamma}. Let us next consider the derivative $\gamma'(r)$. We first observe that, by representation \eqref{covofU}, the matrix $\gamma(r)$ is continuously differentiable except at the origin, and we have
$$
\Vert \gamma'(r)\Vert =\mathcal{O}\left(\max\left(t^{2H_{min}-1},1\right)\right)
$$
as $r\to 0$. To prove \eqref{eq:fbm-gamma-derivative}, by a standard application of dominated convergence theorem, we express $\gamma'(r)$ as the sum of terms
\begin{small}
\begin{align}
&\frac{\bh}{2} \int_{-\infty}^ {-g(r)} \int_{-\infty}^ {-g(r)} e^{\bh x} \left( v'(x+r) - v'(r) + v'(r-s) - v'(x+r-s)\right) e^{\bh s} \bh dx ds  \label{dpartone}\\
&\frac{\bh}{2} \int_{-g(r)}^ {0} \int_{-g(r)}^ {0} e^{\bh x} \left( v'(x+r) - v'(r) + v'(r-s) - v'(x+r-s)\right) e^{\bh s} \bh dx ds. \label{dparttwo}
\end{align}
\end{small}
Now both terms can be treated as \eqref{partone} and \eqref{parttwo}, which concludes the proof.
\end{proof}
\subsection*{Acknowledgements}
Pauliina Ilmonen and Lauri Viitasaari thank Vilho, Yrj\"o ja Kalle V\"ais\"al\"a foundation for financial support. Marko Voutilainen thanks Magnus Ehrnrooth foundation for financial support. Pauliina Ilmonen, Lauri Viitasaari, and Marko Voutilainen thank University of Valparaiso for hospitality.

\bibliographystyle{plain}
\bibliography{pipliateekki}
\end{document}